\documentclass[11pt, letterpaper,final]{amsart}
\usepackage{amsfonts,amsmath, amsthm, amssymb, latexsym, epsfig}
\usepackage[english]{babel}
\usepackage[utf8]{inputenc}
\usepackage[all]{xy}
\usepackage{xspace}
\usepackage{amscd}
\usepackage{comment}
\usepackage{setspace}
\usepackage{enumerate}
\usepackage{stmaryrd}
\usepackage{xcolor}
\usepackage{tikz-cd}
\usepackage{hyperref}
\usepackage[mathscr]{eucal}
\usepackage[notcite,notref]{showkeys}

\usepackage[hmargin=3cm,vmargin=3cm]{geometry}

\newcommand{\A}{{\mathcal{A}}}
\newcommand{\B}{{\mathcal{B}}}
\newcommand{\C}{{\mathcal{C}}}

\renewcommand{\H}{{\mathcal{H}}}
\newcommand{\I}{{\mathcal{I}}}
\newcommand{\J}{{\mathcal{J}}}
\newcommand{\K}{{\mathcal{K}}}
\renewcommand{\L}{{\mathcal{L}}}
\newcommand{\M}{{\mathcal{M}}}

\newcommand{\bbB}{{\mathbb{B}}}
\newcommand{\bbC}{{\mathbb{C}}}

\newcommand{\bbK}{{\mathbb{K}}}
\newcommand{\bbN}{{\mathbb{N}}}

\newcommand{\ran}{\operatorname{ran}}
\newcommand{\id}{\operatorname{id}}
\newcommand{\lag}{\operatorname{lag}}

\newtheorem{thm}{Theorem}[section]
	\newtheorem{cor}[thm]{Corollary}
	\newtheorem{lem}[thm]{Lemma}
	\newtheorem{prp}[thm]{Proposition}

	\newtheorem{theoremintro}{Theorem}

\theoremstyle{definition}
	\newtheorem{dfn}[thm]{Definition}
	
	\newtheorem{asmp}[thm]{Assumption}
	\newtheorem{rmk}[thm]{Remark}
	\newtheorem{exm}[thm]{Example}

\begin{document}
	\title[Shift equivalence relations and C*-correspondences]{Shift equivalence relations through \\ the lens of C*-correspondences}
	
	        \author{Boris Bilich}
        \address{Department of Mathematics, University of Haifa, Haifa, Israel, and Department of Mathematics, University of G\"ottingen, G\"ottingen, Germany.}
        \email{bilichboris1999@gmail.com}
        
        \author{Adam Dor-On}
        \address{Department of Mathematics, University of Haifa, Mount Carmel, Haifa 3103301, Israel.}
        \email{adoron.math@gmail.com}
				
				\author{Efren Ruiz}
        \address{Department of Mathematics\\University of Hawaii,
Hilo\\200 W. Kawili St.\\
Hilo, Hawaii\\
96720-4091 USA}
        \email{ruize@hawaii.edu}
        \date{\today}

        
\begin{abstract}
We continue the study of shift equivalence relations from the perspective of C*-bimodule theory. We study emerging shift equivalence relations following work of the second-named author with Carlsen and Eilers, both in terms of adjacency matrices and in terms of their C*-correspondences, and orient them when possible. In particular, we show that if two regular C*-correspondences are strong shift equivalent, then the intermediary C*-correspondences realizing the equivalence may be chosen to be regular. This result provides the final missing piece in answering a question of Muhly, Pask and Tomforde, and is used to confirm a conjecture of Kakariadis and Katsoulis on shift equivalence of C*-correspondences.  
\end{abstract}

\thanks{B. Bilich was partially supported by a Bloom PhD scholarship at Haifa University. A. Dor-On was partially supported by an NSF-BSF grant no. 2530543 / 2023695 (respectively), a Horizon Marie-Curie SE project no. 101086394 and a DFG Middle-Eastern collaboration project no. 529300231. E. Ruiz was partially supported by the Simons Foundation and by the Mathematisches Forschungsinstitut Oberwolfach.}

\maketitle

\section{Introduction}

Shifts of finite type (SFTs), which are sometimes called topological Markov chains, are fundamental building blocks of symbolic dynamics, and have applications to hyperbolic dynamics, ergodic theory, matrix theory and other areas \cite{RB75, DGS76}.  In a foundational 1973 paper \cite{Wil73}, Williams recast conjugacy and eventual conjugacy problems for SFTs purely in terms of equivalence relations between adjacency matrices of directed graphs. These are called strong shift equivalence (SSE) and shift equivalence (SE), respectively. It was shown by Kim and Roush \cite{KRdeci, KRdecii} that shift equivalence is decidable, but the problem of decidability of SSE remains a fundamental open problem in symbolic dynamics. Williams expected SSE and SE to coincide, but after around 25 years the last hope for a positive answer was extinguished by Kim and Roush \cite{KR99}.

\begin{dfn}
Let $\mathsf{A}$ and $\mathsf{B}$ be matrices indexed by $V$ and $W$ respectively, with entries in $\bbN$ and let $m$ be a positive integer.  

We say that $\mathsf{A}$ and $\mathsf{B}$ are \emph{shift equivalent with $\lag$ $m$} if there are matrices $\mathsf{R}$ over $V \times W$ and $\mathsf{S}$ over $W\times V$ with entries in $\bbN$ such that 
\begin{align*}
\mathsf{A}^m &= \mathsf{R}\mathsf{S}  & \mathsf{S}\mathsf{R} &= \mathsf{B}^m \\
\mathsf{A}\mathsf{R} &=\mathsf{R}\mathsf{B} & \mathsf{B}\mathsf{S} &= \mathsf{A}\mathsf{R}.
\end{align*}

We say that $\mathsf{A}$ and $\mathsf{B}$ are \emph{strong shift equivalent with $\lag$ $m$} if there are matrices $\mathsf{R}_1,..., \mathsf{R}_m$ and $\mathsf{S}_1,...,\mathsf{S}_m$ such that

\begin{align*}
\mathsf{A} &= \mathsf{R}_1\mathsf{S}_1 & \mathsf{S}_m\mathsf{R}_m &= \mathsf{B} 
\end{align*}
$$
\mathsf{S}_i \mathsf{R}_i = \mathsf{R}_{i+1} \mathsf{S}_{i+1} \ \ \ \ \ \text{for all} \ \ \ \ \  i=1,...,m-1.
$$
\end{dfn}

The connection between symbolic dynamics and C*-algebras has a long and fruitful history.  Initiated by Cuntz and Krieger \cite{CK80, Cu81} in which they expressed several equivalence relations between SFTs through associated C*-algebras. In particular, they showed that SSE of finite essential matrices $\mathsf{A}$ and $\mathsf{B}$ implies that the Cuntz-Krieger C*-algebras $\mathcal{O}_\mathsf{A}$ and $\mathcal{O}_\mathsf{B}$ are stably isomorphic in a way preserving their gauge actions $\gamma^\mathsf{A}$ and $\gamma^\mathsf{B}$ together with their diagonal subalgebras $\mathcal{D}_\mathsf{A}$ and $\mathcal{D}_\mathsf{B}$.  Decades later, Carlsen and Rout \cite{CR} completed this connection by proving that SSE is characterized by the stabilized Cuntz-Krieger C*-algebra together with the gauge action and the diagonal subalgebra.  

Following work of the second author with Carlsen and Eilers \cite{CDE}, the problem of encapsulating SE in terms of Cuntz-Krieger C*-algebras was recently resolved by the authors \cite{BDR+}, where we show that SE of finite essential matrices $\mathsf{A}$ and $\mathsf{B}$ is equivalent to $\mathcal{O}_\mathsf{A}$ and $\mathcal{O}_\mathsf{B}$ being stably homotoplically equivalent in a way preserving their gauge actions $\gamma^\mathsf{A}$ and $\gamma^\mathsf{B}$. The study initiated in \cite{CDE} has led to newly emerging shift equivalence relations on the set of essential matrices with entries in $\mathbb{N}$ which are called compatible, aligned, and balanced shift equivalence (see Definition~\ref{dfn-relation-matrices}). One of the main result in \cite{CDE} is that two essential matrices with entries in $\mathbb{N}$ are compatible shift equivalent if and only if they are SSE. It was left open whether aligned and balanced shift relations also characterizes SSE. Our first main result shows that aligned and balanced shift equivalence both coincide with SSE as well.

\begin{theoremintro}\label{thm_cab}
The relations on finite essential matrices with entries in $\mathbb{N}$ given by aligned, balanced and compatible shift equivalence of lag $m$ coincide. Consequently, aligned and balanced shift equivalence relations coincide with SSE.
 \end{theoremintro}
 
A consequence of Theorem~\ref{thm_cab} is that in order to determine whether SSE is decidable it is equivalent to determine whether one of aligned shift equivalence or balanced shift equivalence is decidable. It is unknown to the authors whether these new relations are decidable, but algorithms implementing aligned shift equivalence for a fixed lag perform better than those for some of the other, more complicated-seeming equivalence relations.  As aligned, balanced, and compatible shift equivalence characterizes SSE, they may prove useful in answering open questions to determine when SE matrices are SSE. For example, it is known that $\mathsf{A}_k$ and $\mathsf{B}_k$ are SE for all $k \geq 1$, where
\[
\mathsf{A}_k = \begin{bmatrix} 1 & k \\ k-1 & 1 \end{bmatrix} \quad \text{and}  \quad \mathsf{B}_k = \begin{bmatrix} 1 & k (k-1) \\ 1 & 1 \end{bmatrix}
\]
but whether $\mathsf{A}_k$ and $\mathsf{B}_k$ are SSE for all $k \geq 4$ remains open.  When $k=3$, Baker showed that $\mathsf{A}_3$ and $\mathsf{B}_3$ are SSE (see \cite[Example~7.3.12]{LM95}).

The emergent shift equivalence relations discussed above were discovered through the lens of C*-correspondences, which are used to define ``quantized" analogues of the relations. Given any matrix with entries in $\mathbb{N}$ we may naturally associate a C*-correspondence (sometimes called a graph C*-correspondence) to the matrix in such a way that matrix multiplication translates to taking the interior tensor product of C*-correspondences (see Section \ref{sec_prelim}).  

Inspired by this natural association, Muhly, Pask, and Tomforde \cite{MPT08} enhanced the connection between SFTs and operator algebras by introducing the notion of SE and SSE for C*-correspondences. In their paper \cite[Remark 5.5]{MPT08} they ask whether SSE of two C*-correspondences $X$ and $Y$ implies the SSE of their so-called ``Pimsner dilations" $X_{\infty}$ and $Y_{\infty}$, as another mean of proving that when two C*-correspondences are SSE, then their Cuntz-Pimsner algebras are strong Morita equivalent. This question was taken up by Kakariadis and Katsoulis \cite{KK-JFA2014}, who claimed a positive solution. However, in their \cite[Theorem 5.5]{KK-JFA2014} there is a gap where it is assumed \emph{with} loss of generality that the intermediary C*-correspondences realizing SSE are all regular. Our second main result bridges this gap.

\begin{theoremintro} \label{thm:to-reg-or-full}
Let $X$ and $Y$ be regular (or both regular and full) C*-correspondences, and suppose that there are C*-correspondences $R_1,...,R_m$ and $S_1,...,S_m$ such that
\begin{align*}
X & \cong R_1\otimes S_1 & S_m \otimes R_m & \cong Y 
\end{align*} 
$$
S_i \otimes R_i \cong R_{i+1} \otimes S_{i+1} \ \ \ \ \ \text{for all} \ \ \ \ \  i=1,...,m-1.
$$
Then $R_1,...,R_m$ and $S_1,...,S_m$ can be chosen so that $S_i \otimes R_i$ (and therefore $R_{i+1} \otimes S_{i+1}$) are regular (or both regular and full, respectively) for every $i=1,...,m-1$.
\end{theoremintro}

This provides the final missing ingredient to resolve the question posed by Muhly, Pask and Tomforde in \cite[Remark 5.5]{MPT08}, and in Remark \ref{rmk:equiv-stab-iso} we sketch a proof of this when the coefficient C*-algebras are $\sigma$-unital, and the C*-correspondences are regular and full.

As an application of Theorem \ref{thm:to-reg-or-full} we are also able to resolve the so-called ``shift equivalence problem for C*-correspondences" posed by Kakariadis and Katsoulis in \cite[Section 6]{KK-JFA2014}. More precisely, the shift equivalence problem for C*-correspondences asks if SE of two regular C*-correspondences implies SSE of the C*-correspondences. If we specialize to finite essential matrices with entries in $\mathbb{N}$, the second author with Carlsen and Eilers \cite[Proposition 3.5]{CDE} proved that SE of the C*-correspondences associated to the matrices is equivalent to SE of the underlying matrices. Our third main result applies Theorem \ref{thm:to-reg-or-full} to obtain the SSE analogue.

\begin{theoremintro} \label{thm:msse=sse}
The relations on finite essential matrices with entries in $\mathbb{N}$ given by SSE and SSE of their associated C*-correspondences are equal.
 \end{theoremintro}

As a consequence, and through the Kim--Roush counterexample to Williams' problem \cite{KR99}, in Corollary \ref{cor:sep-neg} we show that SSE for regular and full C*-correspondences is strictly stronger than SE for them. Thus, the shift equivalence problem for C*-corresponden\-ces has a negative answer. Finally, as another application of our results we show that SSE of C*-correspondences of adjacency matrices is not equivalent to them being strong shift equivalent with lag $m=1$, answering yet another question of Kakariadis and Katsoulis from \cite{KK-JFA2014}.

The paper is organized as follows. In Section~\ref{sec_prelim}, we give the definitions of the main objects of the paper. In Section~\ref{sec_eq_CDE} we define aligned and balanced shift relations as they appeared in \cite{CDE} and prove that these coincide with SSE. In Section~\ref{sec-SSE-cor} we prove that SSE of regular C*-correspondences can be implemented via regular C*-correspondences, and explain how this fixes gaps in the literature. Finally, in Section~\ref{sec-modSSE-implies-SSE} we show that for finite essential matrices, SSE of their associated graph C*-correspondences coincides with SSE of the matrices, and explain how this resolves the shift equivalence problem for C*-correspondence in the negative.

\vspace{6pt}

\textbf{Acknowledgments.}
The authors are grateful to Toke Meier Carlsen and S\o ren Eilers for sharing notes on compatible shift equivalence, as well as resulting computer experiments on aligned shift equivalence. The second-named author is also grateful to Idan Pazi for his participation in a research exposure week for undergraduate students at the Technion in September 2022 where open questions on aligned shift equivalence were explored.
 
\section{Preliminary}\label{sec_prelim}

In this section we recall some of the necessary theory. We recommend \cite{LM95} for the basics of the theory of SFTs and \cite{Lan95} for the theory of C*-correspondences. 

\begin{dfn}
Let $\mathsf{A}$ and $\mathsf{B}$ be square matrices indexed by sets $V$ and $W$ respectively, with entries in $\bbN$ and let $m$ be a positive integer.  We say that $\mathsf{A}$ and $\mathsf{B}$ are \emph{shift equivalent with $\lag$ $m$} if there are matrices $\mathsf{R}$ over $V \times W$ and $\mathsf{S}$ over $W\times V$ with entries in $\bbN$ such that 
\begin{align*}
\mathsf{A}^m &= \mathsf{R}\mathsf{S}  & \mathsf{S}\mathsf{R} &= \mathsf{B}^m \\
\mathsf{A}\mathsf{R} &=\mathsf{R}\mathsf{B} & \mathsf{B}\mathsf{S} &= \mathsf{A}\mathsf{R}.
\end{align*}

We say that $\mathsf{A}$ and $\mathsf{B}$ are \emph{elementary shift related} if they are shift equivalent with lag $m=1$, and that they are \emph{strong shift equivalent with lag $m$} if it takes $m$ steps to get from $\mathsf{A}$ to $\mathsf{B}$ in the transitive closure of the elementary shift relation.
\end{dfn}

A directed graph $G=(V,E,s,r)$ is composed of a set of vertices $V$ and a set of edges $E$, along with source and range maps $s,r: E \rightarrow V$. We say that $G$ is finite if $V$ and $E$ are \emph{finite}. In what follows we identify a given cardinality $\kappa$ with the smallest ordinal of cardinality $\kappa$ whenever order is concerned. This ensures that the cardinality of the set of ordinals $\alpha$ for which $0 \leq \alpha < \kappa$ is exactly $\kappa$. Following \cite{CDE}, for a $V \times W$ matrix $\mathsf{F} = [\mathsf{F}_{ij}]$ with cardinal entries, we denote
$$
E_\mathsf{F} := \{ (v, \alpha, w) \ \mid \  0 \leq \alpha < \mathsf{F}_{vw}, v \in V, w \in W\},
$$ 
and similarly define $s: E_F \rightarrow V$ and $r: E_F \rightarrow W$ via $s(v, \alpha, w) = v$ and $r(v, \alpha, w) = w$ respectively. We will say that $\mathsf{F}$ is \emph{essential} if it has no zero rows and no zero columns. When $V = W$, this makes $G_\mathsf{F} := (V, E_\mathsf{F}, r, s )$ into a directed graph in its own right.  We call a collection of matrices $\{\mathsf{ A}_i \}_{i=1}^m$ \emph{admissible} provided that the indexing set for the columns $\mathsf{A}_i$ and the indexing set for the rows of $\mathsf{A}_{i+1}$ are equal.  For an admissible collection of matrices $\{ \mathsf{A}_i \}_{i=1}^m$, we denote the \emph{fibered product}
$$
\prod_{i=1}^m E_{\mathsf{A}_i} =E_{\mathsf{A}_1} \times  E_{\mathsf{A}_2} \times \cdots \times E_{\mathsf{A}_m}:= \{ (e_1,e_2, \dots , e_m) \ \mid e_i \in E_{\mathsf{A}_i}, r(e_i)=s(e_{i+1}) \}.
$$
We will also denote $(e_1,e_2, \dots , e_m) \in\prod_{i=1}^m E_{\mathsf{A}_i}$ by $e_1 e_2 \cdots e_m$.  Thus, elements of $\prod_{i=1}^m E_{\mathsf{A}_i}$ can be thought as a concatenation of edges.  When $V= W$, $E_\mathsf{C}^n$ will denoted the $n$-fold product of $E_\mathsf{C}$ which can be naturally identified with $E_{\mathsf{C}^n}$.

Let $\A$ be a C*-algebra. A right Hilbert $\A$-module $X$ is a linear space with a right action of the C*-algebra $\A$ and an $\A$-valued inner product $\langle \cdot, \cdot \rangle$ such that $X$ is complete with respect to the induced norm $\| \xi \|:= \| \langle \xi,\xi \rangle \|^{1/2}$ for $\xi \in X$.

For a Hilbert $\A$-module $X$ we denote by $\L(X)$ the C*-algebra of all adjointable operators on $X$. For $\xi,\eta \in X$, the rank-one operator $\theta_{\xi,\eta} \in \L(X)$ is defined by $\theta_{\xi,\eta}(\zeta) = \xi \langle \eta,\zeta \rangle$ for $\zeta \in X$. Then, the ideal of \emph{generalized compact operators} $\K(X)$ in $\L(X)$ is the closed linear span of elements of the form $\theta_{\xi,\eta}$ for $\xi,\eta \in X$.

Now let $\B$ be another C*-algebra. An \emph{$\A-\B$ correspondence} is a right Hilbert $\B$-module $X$ together with a $*$-representation $\phi_X : \A \rightarrow \L(X)$. When we have an $\A-\B$ correspondence, we will often treat it as an $\A-\B$ bimodule. We will say that $X$ is \emph{full} if $\overline{\mathrm{span}} \{ \langle \xi, \eta \rangle : \xi, \eta \in X \} = \B$, \emph{injective} if $\phi_X$ is an injection and \emph{proper} if $\phi_X$ has image in $\K(X)$. Moreover, we will say that $X$ is \emph{regular} if it is both injective and proper and \emph{irreducible} if for any non-trivial closed right $\B$-sub-bimodule $Y \subseteq X$ we have $Y=X$.

\begin{asmp} \label{a:nd}
We assume throughout the paper that correspondences are non-degenerate in the sense that $\phi_X(\A)X = X$, where $\phi_X(\A)X$ denotes the \emph{closure} of the linear span of elements of the form $\phi_X(a)\xi$ for $a\in \A$ and $\xi \in X$. 
\end{asmp}

We will say that two $\A-\B$ correspondences $X$ and $Y$ are \emph{unitarily isomorphic}, and denote this by $X\cong Y$, if there is an isometric bimodule surjection $\tau: X \rightarrow Y$, and call $\tau$ a \emph{unitary correspondence isomorphism}. Due to the polarization identity, a surjective bimodule map $\tau$ is a correspondence isomorphism if and only if for $\xi,\eta \in X$ we have
$$
\langle \tau(\xi),\tau(\eta) \rangle = \langle \xi,\eta \rangle.
$$

For an $\A-\B$ correspondence $X$ and a $\B-\C$ correspondence $Y$, we may define the interior tensor product $\A -\C$ correspondence $X \otimes_{\B} Y$ as follows. Let $X \odot_{\B} Y$ be the quotient of the algebraic tensor product by the subspace generated by elements of the form
$$
\xi b \otimes \eta - \xi \otimes b \eta, \ \ \text{for} \ \ \xi \in X, \ \eta \in Y, \ b \in \B.
$$
We define a $\C$-valued inner product and a left $\A$ action by setting
$$
\langle \xi \otimes \eta, \xi' \otimes \eta' \rangle = \langle \eta , \langle \xi ,\xi' \rangle \eta' \rangle, \ \ \text{for} \ \ \xi,\xi' \in X, \ \ \eta,\eta' \in Y
$$
$$
a \cdot (\xi \otimes \eta) = (a\xi) \otimes \eta, \ \ \text{for} \ \ \xi \in X, \ \ \eta \in Y, \ \ a \in \A,
$$
and denote by $X \otimes_{\B} Y$ the separated completion of $X \odot_{\B} Y$ with respect to the $\C$-valued semi-inner product above. It then follows that $X\otimes_{\B} Y$ is a $\A-\C$ correspondence. 

\begin{dfn}
Let $\A$ and $\B$ be C*-algebras.  A \emph{Hilbert $\A-\B$ bimodule} $X$ is a complex vector space which is both a right Hilbert $\B$-module and a left Hilbert $\A$-module, such that 
\begin{enumerate}
\item The left action of $\B$ is by adjointables on the right Hilbert $\A$-module structure, and the right action of $\A$ is by adjointables on the left Hilbert $\B$-module structure.

\item for all $x, y, z \in X$, 
$$
_\A \langle x, y \rangle \cdot z = x \cdot \langle y, z \rangle_\B.
$$
\end{enumerate}
If additionally we have that the right and left Hilbert C*-module structures on $X$ are full, we will say that $X$ is an \emph{imprimitivity $\A-\B$ bimodule}
\end{dfn}

Next, let $V$ and $W$ be sets, and denote by $\A = c(V)$ and $\B = c(W)$ the commutative C*-algebras consisting of functions on $V$ and $W$ respectively. Let $\mathsf{R}$ be a $V\times W$ matrix with cardinal entries. We define an $\A-\B$ correspondence $X(\mathsf{R})$ from $E_\mathsf{R}$ as follows.
Let $X(\mathsf{R})$ be the space of functions from $E_\mathsf{R}$ to $\bbC$.
The bimodule structure on $X(\mathsf{R})$ is defined by
\[
  (a\cdot \xi \cdot b)(e) = a(r(e))\xi(e) b(s(e))
\]
for all $a\in \A$, $b\in \B$, $\xi \in X(\mathsf{R})$, and $e\in E_\mathsf{R}$.
The $\B$-valued inner product on $X(\mathsf{R})$ is given by
\[
  \langle \xi, \eta\rangle(w) = \sum_{s(e)=w}\overline{\xi(e)}\eta(e)
\]
for all $\xi,\eta\in X(\mathsf{R})$ and $w\in W$. With this $c(W)$-valued inner product and $c(V)$ left action, $X(\mathsf{R})$ becomes a $\A-\B$ correspondence.
When the matrix $\mathsf{R}$ has entries in $\mathbb{N}$, is essential and has finite columns, the correspondence $X(\mathsf{R})$ is both full and regular. If $\mathsf{A}$ is a $V\times V$ matrix with cardinal entries, we say that $X(\mathsf{A})$ is the \emph{graph correspondence} of $G_{\mathsf{A}}$.

In Section \ref{sec-modSSE-implies-SSE} we will greatly generalize this construction to include coefficient C*-algebras that arbitrary subalgebras of compact operators. By standard representation theory of compact operators, these are always of the form $\A = \oplus_{v \in V} \bbK(\H_v)$ and $\B = \oplus_{w\in W} \bbK(\H_w)$, where by $\mathbb{K}(\H)$ we mean compact operators on the Hilbert space $\H$, and $V$ and $W$ are arbitrary sets. When the Hilbert spaces $\H_v$ and $\H_w$ are all $1$-dimensional and $R$ has cardinal entries, we will recover our previous construction because we then get that $\A \cong c(V)$ and $\B \cong c(W)$. 

We will often think of $X(\mathsf{R})$ as a sort of ``quantization" of the matrix $\mathsf{R}$. This is justified by \cite[Proposition 3.3]{CDE}, which shows that if $\mathsf{R}$ is a $V\times W$ matrix and $\mathsf{S}$ is a $W\times U$ matrix, both with cardinal entries, then there is a unitary isomorphism of correspondences
\[
    X(\mathsf{R})\otimes_{c(W)} X(\mathsf{S})\cong X(\mathsf{R}\mathsf{S}).
\]
This will also be generalized in Section \ref{sec-modSSE-implies-SSE} to our construction of graph correspondences when the underlying C*-algebras are subalgebras of compact operators.

\section{Balanced and Aligned shift equivalence for matrices}\label{sec_eq_CDE}

In this section we discuss two new equivalence relations for adjacency matrices which are called aligned and balanced shift relations. These have appeared in \cite{CDE} and are are motivated from the study of compatible shift equivalence, especially from the proof of its transitivity. Surprisingly, in the context of adjacency matrices, we are able to show that both aligned and balanced shift relations coincide with strong shift equivalence.

\begin{dfn}[{\cite[Definition~1.2]{CDE}}]\label{dfn-path-iso}
Let $\{ \mathsf{A}_i \}_{i=1}^m$ and $\{ \mathsf{B}_i \}_{i=1}^n$ be two collections of admissible matrices.  A \emph{path isomorphism} is a bijection $\varphi \colon\prod_{i=1}^m E_{\mathsf{A}_i} \to \prod_{i=1}^n E_{\mathsf{B}_i}$ such that $s( \varphi(e_1 \cdots e_m)) = s(e_1 \cdots e_m)$ and $r( \varphi(e_1 \cdots e_m))=r(e_1 \cdots e_m)$ for all $e_1 \cdots e_m \in \prod_{i=1}^m E_{\mathsf{A}_i}$.
\end{dfn}

Suppose $\mathsf{A}$ and $\mathsf{B}$ are shift equivalent with $\lag$ $m$ via matrices $\mathsf{R}$ over $V \times W$ and $\mathsf{S}$ over $W \times V$ with entries in $\bbN$.  Then it is clear that there are path isomorphisms
\begin{align*}
    \varphi_\mathsf{R} \colon E_\mathsf{A} \times E_\mathsf{R} \to E_\mathsf{R} \times E_\mathsf{B}, \quad & \varphi_\mathsf{S} \colon E_\mathsf{B} \times E_\mathsf{R} \to E_\mathsf{S} \times E_\mathsf{A} \\
    \psi_\mathsf{A} \colon E_\mathsf{R} \times E_\mathsf{S} \to E_\mathsf{A}^m, \quad & \psi_\mathsf{B} \colon E_\mathsf{S} \times E_\mathsf{R} \to E_\mathsf{B}^m.
\end{align*}
Moreover, if there are path isomorphisms 
\begin{align*}
    \varphi_\mathsf{R} \colon E_\mathsf{A} \times E_\mathsf{R} \to E_\mathsf{R} \times E_\mathsf{B}, \quad & \varphi_\mathsf{S} \colon E_\mathsf{B} \times E_\mathsf{S} \to E_\mathsf{S} \times E_\mathsf{A} \\
    \psi_\mathsf{A} \colon E_\mathsf{R} \times E_\mathsf{S} \to E_\mathsf{A}^m, \quad & \psi_\mathsf{B} \colon E_\mathsf{S} \times E_\mathsf{R} \to E_\mathsf{B}^m,
\end{align*}
then clearly
\begin{align*}
\mathsf{A}\mathsf{R} &=\mathsf{R}\mathsf{B} & \mathsf{B}\mathsf{S} &= \mathsf{S}\mathsf{A} \\ \mathsf{A}^m &= \mathsf{R}\mathsf{S}  & \mathsf{S}\mathsf{R} &= \mathsf{B}^m.
\end{align*}

New notions of shift equivalence emerge by imposing relations between path isomorphisms.

\begin{dfn}
Let $\mathsf{A}$ and $\mathsf{B}$ be matrices indexed by $V$ and $W$ respectively, with entries in $\bbN$ and let $m$ be a positive integer. We say that the tuple $(\mathsf{R}, \mathsf{S}, \varphi_\mathsf{R}, \varphi_\mathsf{S}, \psi_\mathsf{A}, \psi_\mathsf{B})$ is a \emph{concrete shift between $\mathsf{A}$ and $\mathsf{B}$ with $\lag$ $m$} if $\mathsf{R}$ is a matrix over $V \times W$ with entries in $\bbN$, $\mathsf{S}$ is a matrix over $W\times V$ with entries in $\bbN$, and each of the maps
\begin{align*}
    \varphi_\mathsf{R} \colon E_\mathsf{A} \times E_\mathsf{R} \to E_\mathsf{R} \times E_\mathsf{B}, \quad & \varphi_\mathsf{S} \colon E_\mathsf{B} \times E_\mathsf{S} \to E_\mathsf{S} \times E_\mathsf{A} \\
    \psi_\mathsf{A} \colon E_\mathsf{R} \times E_\mathsf{S} \to E_\mathsf{A}^m, \quad & \psi_\mathsf{B} \colon E_\mathsf{S} \times E_\mathsf{R} \to E_\mathsf{B}^m
\end{align*}
are path isomorphisms.
\end{dfn}

Let $\mathsf{A}$ and $\mathsf{B}$ be matrices indexed by $V$ and $W$ respectively, with entries in $\bbN$.  Suppose there exists a matrix $\mathsf{R}$ with entries in $\bbN$ and suppose there exists a path isomorphism $\varphi \colon E_\mathsf{A} \times E_\mathsf{R} \to E_\mathsf{R} \times E_\mathsf{B}$.  For any positive integer $m$ we denote by $\varphi^{(m)} : E_\mathsf{A}^m \times E_\mathsf{R} \rightarrow E_\mathsf{R} \times E_\mathsf{B}^m$ the path isomorphism defined by 
$$
\varphi^{(m)} = (\varphi \times \id_{\mathsf{B}^{m-1}} )(\id_\mathsf{A} \times \varphi \times \id_{\mathsf{B}^{m-2}} ) \cdots (\id_{\mathsf{A}^{m-1}} \times \varphi).
$$

\begin{dfn}\label{dfn-relation-matrices}
Let $\mathsf{A}$ and $\mathsf{B}$ be matrices indexed by $V$ and $W$ respectively, with entries in $\bbN$ and let $(\mathsf{R},\mathsf{S}, \varphi_\mathsf{R}, \varphi_\mathsf{S}, \psi_\mathsf{A}, \psi_\mathsf{B})$ be a concrete shift between $\mathsf{A}$ and $\mathsf{B}$ with $\lag$ $m$. We say that $(\mathsf{R}, \mathsf{S}, \varphi_\mathsf{R}, \varphi_\mathsf{S}, \psi_\mathsf{A}, \psi_\mathsf{B})$ is
\begin{enumerate}
\item \emph{aligned} if 
\begin{align*}
    (\psi_\mathsf{A} \times \id_\mathsf{A})(\id_\mathsf{R} \times \varphi_\mathsf{S})(\varphi_\mathsf{R} \times \id_\mathsf{S}) &= (\id_\mathsf{A} \times \psi_\mathsf{A}) \\
    (\psi_\mathsf{B} \times \id_\mathsf{B})(\id_\mathsf{S} \times \varphi_\mathsf{R})(\varphi_\mathsf{S} \times \id_\mathsf{R}) &= (\id_\mathsf{B} \times \psi_\mathsf{B}).
\end{align*}
		
\item \emph{balanced} if 
\begin{align*}
\psi_\mathsf{A}^{-1} \times \psi_\mathsf{A} &= (\id_\mathsf{R} \times \varphi_\mathsf{S}^{(m)}) ( \varphi_\mathsf{R}^{(m)} \times \id_\mathsf{S})  \\
\psi_\mathsf{B}^{-1} \times \psi_\mathsf{B} &= (\id_\mathsf{S} \times \varphi_\mathsf{R}^{(m)}) ( \varphi_\mathsf{S}^{(m)} \times \id_\mathsf{R}).
\end{align*}
		
\item \emph{compatible} if 
\begin{align*}
\varphi_\mathsf{R}^{(m)} &= (\id_\mathsf{R} \times \psi_\mathsf{B})(\psi_\mathsf{A}^{-1} \times \id_\mathsf{R} ) \\  
\varphi_\mathsf{S}^{(m)} &= (\id_\mathsf{S} \times \psi_\mathsf{A})(\psi_\mathsf{B}^{-1} \times \id_\mathsf{S} ).
\end{align*}

\end{enumerate}
We say that $\mathsf{A}$ and $\mathsf{B}$ are \emph{aligned /  balanced / compatible} shift related with $\lag$ $m$ if there exists a concrete shift between $\mathsf{A}$ and $\mathsf{B}$ with $\lag$ $m$ which is aligned /  balanced / compatible respectively.
\end{dfn}

Note that we implicitly use the ``associativity" path isomorphism $E_\mathsf{A} \times E_\mathsf{A}^m \cong E_\mathsf{A}^m \times E_\mathsf{A}$ given by $( e, (e_1, e_2, \ldots, e_m)) = ( (e_1, e_2, \ldots, e_{m-1}) , e_m )$ as this is the canonical choice when identifying $E_\mathsf{A} \times E_\mathsf{A}^m$ and $E_\mathsf{A}^m \times E_\mathsf{A}$ with $E_\mathsf{A}^{m+1}$ via path isomorphisms.

By \cite[Proposition~4.4]{CDE}, compatible shift relation on the set of essential matrices with entries in $\mathbb{N}$ is an equivalence relation. Moreover, by \cite[Theorem~1.3]{CDE} strong shift equivalence coincides with compatible shift equivalence.  One of the consequence of our proof is that aligned, balanced, and compatible shift relations coincide on essential matrices with entries in $\bbN$ is that both aligned and balanced shift relations are equivalence relations as well. 

A consequence of the proof of \cite[Lemma~4.2]{CDE} is that if a concrete shift $(\mathsf{R}, \mathsf{S}, \varphi_\mathsf{R}, \varphi_\mathsf{S}, \psi_\mathsf{A}, \psi_\mathsf{B})$ is compatible, then it is both aligned and balanced (see \cite[Remark~4.3]{CDE}). The following shows that if a concrete shift is aligned, then it is balanced.

\begin{prp} \label{p:aligned->balanced}
Suppose $\mathsf{A}$ and $\mathsf{B}$ are essential square matrices over $\bbN$, and $\mathsf{R}$ and $\mathsf{S}$ some matrices over $\bbN$. Suppose that $(\mathsf{R}, \mathsf{S}, \varphi_\mathsf{R}, \varphi_\mathsf{S}, \psi_\mathsf{A}, \psi_\mathsf{B})$ is a concrete shift between $\mathsf{A}$ and $\mathsf{B}$ with $\lag$ $m$. If the tuple $(\mathsf{R}, \mathsf{S}, \varphi_\mathsf{R}, \varphi_\mathsf{S}, \psi_\mathsf{A}, \psi_\mathsf{B})$ is aligned, then it is balanced. 
\end{prp}

\begin{proof}
We will prove
\begin{equation}\label{eq_align_implies_balanced}
 (\id_\mathsf{R} \times \varphi_\mathsf{S}^{(\ell)}) (\varphi_\mathsf{R}^{(\ell)} \times\id_\mathsf{S})=(\psi_\mathsf{A}^{-1} \times \id_{\mathsf{A}^\ell})(\id_{\mathsf{A}^\ell} \times \psi_\mathsf{A}),
\end{equation}
for all $1 \leq \ell \leq m$.  Since the concrete shift equivalence is aligned, \eqref{eq_align_implies_balanced} holds for $\ell=1$.  

Assume \eqref{eq_align_implies_balanced} holds for some $\ell$.  We will show \eqref{eq_align_implies_balanced} holds for $\ell+1$.  By the definition of $\varphi_\mathsf{R}^{(\ell+1)}$ and $\varphi_\mathsf{S}^{(\ell+1)}$, we have 
\begin{align*}
\varphi_\mathsf{R}^{(\ell+1)} &= (\varphi_\mathsf{R}^{ (\ell) } \times \id_\mathsf{B} ) ( \id_{\mathsf{A}^\ell} \times \varphi_\mathsf{R})  & & \text{and} &  \varphi_\mathsf{S}^{(\ell+1)} &= (\varphi_\mathsf{S}^{ (\ell) } \times \id_\mathsf{A} ) ( \id_{\mathsf{B}^\ell} \times \varphi_\mathsf{S}).
\end{align*}
Let $a_1\ldots a_{\ell+1} rs \in E_{\mathsf{A}^{\ell+1} } \times E_\mathsf{R} \times E_\mathsf{S}$.  Choose $r_{\ell+1}\in E_{\mathsf{R}}$, $b_{\ell+1} \in E_\mathsf{B}$, $a_{\ell+1} ' \in E_\mathsf{A}$,  and $ s_{\ell+1} \in E_\mathsf{S}$ such that 
\begin{align*}
\varphi_\mathsf{R}( a_{\ell+1} r) &= r_{\ell+1} b_{\ell+1}  & \varphi_\mathsf{S} ( b_{\ell+1} s ) &= s_{\ell+1} a_{\ell+1}' \\
\end{align*}
Then 
\begin{align*}
&(\id_\mathsf{R} \times \varphi_\mathsf{S}^{(\ell+1)})(\varphi_\mathsf{R}^{(\ell+1)} \times\id_\mathsf{S}) (a_1\ldots a_{\ell+1} rs )\\
&= (\id_\mathsf{R} \times \varphi_\mathsf{S}^{ (\ell) } \times \id_\mathsf{A} )  (\id_\mathsf{R} \times \id_{\mathsf{B}^\ell } \times \varphi_\mathsf{S} )(\varphi_\mathsf{R}^{ (\ell) } \times \id_\mathsf{B}  \times\id_\mathsf{S}) ( \id_{\mathsf{A}^\ell} \times \varphi_\mathsf{R} \times\id_\mathsf{S}) (a_1\ldots a_{\ell+1} rs ) \\
&= (\id_\mathsf{R} \times \varphi_\mathsf{S}^{ (\ell) } \times \id_\mathsf{A} )  (\id_\mathsf{R} \times \id_{\mathsf{B}^\ell } \times \varphi_\mathsf{S} )(\varphi_\mathsf{R}^{ (\ell) } \times \id_\mathsf{B}  \times\id_\mathsf{S}) (a_1 \cdots a_\ell r_{\ell+1} b_{\ell+1}s) \\
& = (\id_\mathsf{R} \times \varphi_\mathsf{S}^{ (\ell) } \times \id_\mathsf{A} )  ( \varphi_\mathsf{R}^{(\ell)} \times \varphi_\mathsf{S} )  ( a_1 \cdots a_\ell r_{\ell+1} b_{\ell+1}s ) \\
&= (\id_\mathsf{R} \times \varphi_\mathsf{S}^{ (\ell) } \times \id_\mathsf{A} ) (\varphi_\mathsf{R}^{(\ell)} (a_1 \cdots a_\ell r_{\ell+1}) s_{\ell+1} a_{\ell+1}')\\
&= (\id_\mathsf{R} \times \varphi_\mathsf{S}^{ (\ell) } \times \id_\mathsf{A} ) ( \varphi_\mathsf{R}^{(\ell)} \times \id_\mathsf{S} \times \id_\mathsf{A} ) (a_1 \cdots a_\ell r_{\ell+1} s_{\ell+1} a_{\ell+1}' ) \\
&= (\psi_\mathsf{A}^{-1} \times \id_{\mathsf{A}^{\ell}} \times \id_{\mathsf{A}})(\id_{\mathsf{A}^{\ell}} \times \psi_\mathsf{A} \times \id_\mathsf{A})(a_1 \cdots a_\ell r_{\ell+1} s_{\ell+1} a_{\ell+1}') \\
&= (\psi_\mathsf{A}^{-1} \times \id_{\mathsf{A}^{\ell+1}})(\id_{\mathsf{A}^{\ell}} \times \psi_\mathsf{A} \times \id_\mathsf{A}) ( \id_{\mathsf{A}^\ell} \times \id_\mathsf{R} \times \varphi_\mathsf{S}) (\id_{\mathsf{A}^\ell} \times \varphi_\mathsf{R} \times \id_\mathsf{S} ) (a_1\ldots a_{\ell+1} rs )\\
&=  (\psi_\mathsf{A}^{-1} \times \id_{\mathsf{A}^{\ell+1}})(\id_{\mathsf{A}^{\ell}} \times \psi_\mathsf{A} \times \id_\mathsf{A}) ( \id_{\mathsf{A}^\ell} \times \psi_\mathsf{A}^{-1} \times \id_\mathsf{A}) (\id_{\mathsf{A}^\ell} \times \id_\mathsf{A} \times \psi_\mathsf{A} ) (a_1\ldots a_{\ell+1} rs )\\ 
&= (\psi_\mathsf{A}^{-1} \times \id_{\mathsf{A}^{\ell+1}})(\id_{\mathsf{A}^{\ell+1}} \times \psi_\mathsf{A})(a_1\ldots a_{\ell+1} rs ).
\end{align*}

It follows that \eqref{eq_align_implies_balanced} holds for all $1 \leq i \leq m$.  In particular, 
\begin{equation*}
 (\id_\mathsf{R} \times \varphi_\mathsf{S}^{(m)}) (\varphi_\mathsf{R}^{(m)} \times\id_\mathsf{S})(\psi_\mathsf{A}^{-1} \times \id_{\mathsf{A}^m}) = (\id_{\mathsf{A}^m} \times \psi_\mathsf{A})=\psi_\mathsf{A}^{-1} \times \psi_\mathsf{A}.
\end{equation*}
A symmetric argument gives
\begin{equation*}
(\id_\mathsf{S} \times \varphi_\mathsf{R}^{(m)})(\varphi_\mathsf{S}^{(m)}\times\id_\mathsf{R})=\psi_\mathsf{B}^{-1} \times \psi_\mathsf{B}.\qedhere
\end{equation*}
\end{proof}

\begin{thm}\label{thm:balanced->compatible}
Let $\mathsf{A}$ and $\mathsf{B}$ be essential matrices over $V$ and $W$ respectively, with entries in $\bbN$ and let $(\mathsf{R}, \mathsf{S}, \varphi_\mathsf{R}, \varphi_\mathsf{S}, \psi_\mathsf{A}, \psi_\mathsf{B})$ be a concrete shift between $\mathsf{A}$ and $\mathsf{B}$ with $\lag$ $m$. Then the following are equivalent:
\begin{enumerate}
\item $(\mathsf{R}, \mathsf{S}, \varphi_\mathsf{R}, \varphi_\mathsf{S}, \psi_\mathsf{A}, \psi_\mathsf{B})$ is aligned.
\item $(\mathsf{R}, \mathsf{S}, \varphi_\mathsf{R}, \varphi_\mathsf{S}, \psi_\mathsf{A}, \psi_\mathsf{B})$ is balanced.
\item $(\mathsf{R}, \mathsf{S}, \varphi_\mathsf{R}, \varphi_\mathsf{S}, \psi_\mathsf{A}, \psi_\mathsf{B})$ is compatible. 
\end{enumerate}
\end{thm}

\begin{proof}
\cite[Lemma~4.2]{CDE} proves that if $(\mathsf{R}, \mathsf{S}, \varphi_\mathsf{R}, \varphi_\mathsf{S}, \psi_\mathsf{A}, \psi_\mathsf{B})$ is compatible, then it is aligned, and Proposition \ref{p:aligned->balanced} shows that if the tuple $(\mathsf{R}, \mathsf{S}, \varphi_\mathsf{R}, \varphi_\mathsf{S}, \psi_\mathsf{A}, \psi_\mathsf{B})$ is aligned, then it is balanced.

Assume now that $(\mathsf{R}, \mathsf{S}, \varphi_\mathsf{R}, \varphi_\mathsf{S}, \psi_\mathsf{A}, \psi_\mathsf{B})$ is balanced. That is,
\begin{align*}
\psi_\mathsf{A}^{-1} \times \psi_\mathsf{A} &= (\id_\mathsf{R} \times \varphi_\mathsf{S}^{(m)}) ( \varphi_\mathsf{R}^{(m)} \times \id_\mathsf{S}) \quad \text{and} \\
\psi_\mathsf{B}^{-1} \times \psi_\mathsf{B} &= (\id_\mathsf{S} \times \varphi_\mathsf{R}^{(m)}) ( \varphi_\mathsf{S}^{(m)} \times \id_\mathsf{R}).
\end{align*}
We will show that $(\mathsf{R}, \mathsf{S}, \varphi_\mathsf{R}, \varphi_\mathsf{S}, \psi_\mathsf{A}, \psi_\mathsf{B})$ is compatible, that is, we must show
$$
\varphi_\mathsf{R}^{(m)} = (\id_\mathsf{R} \times \psi_\mathsf{B})(\psi_\mathsf{A}^{-1} \times \id_\mathsf{R} ) \quad \text{and} \quad \varphi_\mathsf{S}^{(m)} = (\id_\mathsf{S} \times \psi_\mathsf{A})(\psi_\mathsf{B}^{-1} \times \id_\mathsf{S} ).
$$
We will only show $\varphi_\mathsf{R}^{(m)} = (\id_\mathsf{R} \times \psi_\mathsf{B})(\psi_\mathsf{A}^{-1} \times \id_\mathsf{R} )$ as a symmetric argument shows $\varphi_\mathsf{S}^{(m)} = (\id_\mathsf{S} \times \psi_\mathsf{A})(\psi_\mathsf{B}^{-1} \times \id_\mathsf{S} )$.

First note that since $\mathsf{A}$ and $\mathsf{B}$ are essential matrices and since $\mathsf{R}\mathsf{S}=\mathsf{A}^m$ and $\mathsf{S}\mathsf{R}=\mathsf{B}^m$, each row of $\mathsf{R}$ and each row of $\mathsf{S}$ are nonzero.   We will use this observation throughout the proof.  Let $e_1\cdots e_m r \in E_\mathsf{A}^m \times E_\mathsf{R}$. Let $r' \in E_\mathsf{R}$ and $f_1\cdots f_m \in E_\mathsf{B}^m$ be such that 
\begin{equation}\label{eq:cse1}
    \varphi_\mathsf{R}^{(m)} ( e_1 \cdots e_m r) = r' f_1 \cdots f_m.
\end{equation}
We must show that 
\begin{equation*}
    (\id_\mathsf{R} \times \psi_\mathsf{B}) (\psi_\mathsf{A}^{-1} \times \id_\mathsf{R}) (e_1\cdots e_m r)=r'f_1\cdots f_m.
\end{equation*}
Since each row of $\mathsf{S}$ is nonzero, there exists $s \in E_\mathsf{S}$ such that $e_1 \cdots e_m rs \in E_\mathsf{A}^m \times E_\mathsf{R} \times E_\mathsf{S}$.    Choose $s' \in E_\mathsf{S}$ and $e_1' \cdots e_m' \in E_\mathsf{A}^m$ such that 
\begin{equation}\label{eq:cse2}
    \varphi_\mathsf{S}^{(m)} ( f_1 \cdots f_m s )= s' e_1' \cdots e_m'.
\end{equation}
 Since $\psi_\mathsf{A}^{-1} \times \psi_\mathsf{A}  =(\id_\mathsf{R} \times \varphi_\mathsf{S}^{(m)}) ( \varphi_\mathsf{R}^{(m)} \times \id_\mathsf{S})$, we have 
\begin{align*}
    \psi_\mathsf{A}^{-1}( e_1 \cdots e_m ) \psi_\mathsf{A} (rs) &\overset{\phantom{\eqref{eq:cse1}}}{=} (\id_\mathsf{R} \times \varphi_\mathsf{S}^{(m)})(\varphi_\mathsf{R}^{(m)} \times \id_\mathsf{S}) (e_1\cdots e_m rs) \\
            &\overset{\eqref{eq:cse1}}{=} (\id_\mathsf{R} \times \varphi_\mathsf{S}^{(m)}) ( r'f_1\cdots f_m s)  \\
            &\overset{\eqref{eq:cse2}}{=}r' s' e_1' \cdots e_m'.
\end{align*}
Therefore, 
\begin{align}
    \psi_\mathsf{A}^{-1}( e_1\cdots e_m) &= r's' \label{eq:cse3} \quad \text{and}\\
    \psi_\mathsf{A}(rs)&= e_1'\cdots e_m' \label{eq:cse4}.
\end{align}

We claim that 
\begin{equation}\label{eq:cse5}
    \psi_\mathsf{B}(s'r) = f_1 \cdots f_m.
\end{equation}
Since each row of $\mathsf{R}$ is nonzero, there exists $r'' \in E_\mathsf{R}$ such that $f_1 \cdots f_m s r'' \in E_\mathsf{B}^m \times E_\mathsf{S} \times E_\mathsf{R}$. Let $r''' \in E_\mathsf{R}$ and $f_1'\cdots f_m' \in E_\mathsf{B}^m$ be such that 
\begin{equation}\label{eq:cse6}
\varphi_\mathsf{R}^{(m)}( e_1'\cdots e_m' r'' ) = r''' f_1'\cdots f_m'.
\end{equation}
Since  $\psi_\mathsf{B}^{-1} \times \psi_\mathsf{B} = (\id_\mathsf{S} \times \varphi_\mathsf{R}^{(m)}) ( \varphi_\mathsf{S}^{(m)} \times \id_\mathsf{R})$, 
\begin{align*}
\psi_\mathsf{B}^{-1}( f_1 \cdots f_m) \psi_\mathsf{B}(sr'') &\overset{\phantom{\eqref{eq:cse2}}}{=} (\id_\mathsf{S} \times \varphi_\mathsf{R}^{(m)}) ( \varphi_\mathsf{S}^{(m)} \times \id_\mathsf{R})(f_1\cdots f_m sr'') \\
&\overset{\eqref{eq:cse2}}{=} (\id_\mathsf{S} \times \varphi_\mathsf{R}^{(m)})(s' e_1'\cdots e_m' r'') \\
&\overset{\eqref{eq:cse6}}{=}  s' r''' f_1' \cdots f_m'.
\end{align*}
Therefore, $\psi_\mathsf{B}^{-1}(f_1 \cdots f_m)=s' r'''$.  To finish the proof of the claim that $\psi_\mathsf{B}(s'r)=f_1 \cdots f_m$, we will show $r=r'''$.  Since each row of $\mathsf{S}$ is nonzero, there exists $s'' \in E_\mathsf{S}$ such that $e_1'\cdots e_m' r'' s'' \in E_\mathsf{A}^m \times E_\mathsf{R} \times E_\mathsf{S}$.  Since $\psi_\mathsf{A}^{-1} \times \psi_\mathsf{A} = (\id_\mathsf{R} \times \varphi_\mathsf{S}^{(m)}) ( \varphi_\mathsf{R}^{(m)} \times \id_\mathsf{S})$,
\begin{align*}
rs \psi_\mathsf{A}(r''s'') &\overset{\eqref{eq:cse4}}{=} \psi_\mathsf{A}^{-1}( e_1'\cdots e_m')\psi_\mathsf{A} (r''s'')  \\
                &\overset{\phantom{\eqref{eq:cse6}}}{=}(\id_\mathsf{R} \times \varphi_\mathsf{S}^{(m)} )( \varphi_\mathsf{R}^{(m)} \times \id_\mathsf{S})(e_1'\cdots e_m' r'' s'') \\
                &\overset{\eqref{eq:cse6}}{=}(\id_\mathsf{R} \times \varphi_\mathsf{S}^{(m)} ) (r''' f_1'\cdots f_m' s'')\\
                &=r''' \varphi_\mathsf{R}^{(m)}(f_1'\cdots f_m' s'').
\end{align*}
Therefore, $r=r'''$ which implies $\psi_\mathsf{B}^{-1}(f_1\cdots f_m)=s'r''' = s'r$.  Therefore, $\psi_\mathsf{B}(s'r)=f_1\cdots f_m$ and completing the proof of the claim. 

We now finish the proof that 
$$
\varphi_\mathsf{R}^{(m)}( e_1 \cdots e_m r)=(\id_\mathsf{R} \times \psi_\mathsf{B})(\psi_\mathsf{A}^{-1} \times \id_\mathsf{R})( e_1\cdots e_m r).
$$
Indeed,
\begin{align*}
    \varphi_\mathsf{R}^{(m)} ( e_1 \cdots e_m r ) &\overset{\eqref{eq:cse1}}{=} r' f_1 \cdots f_m \\
                &\overset{\eqref{eq:cse5}}{=} r' \psi_\mathsf{B}(s'r)  \\
                &= (\id_\mathsf{R} \times \psi_\mathsf{B})(r's'r) \\
                &\overset{\eqref{eq:cse3}}{=}(\id_\mathsf{R} \times \psi_\mathsf{B})(\psi_\mathsf{A}^{-1}(e_1\cdots e_m)r) \\
    &=(\id_\mathsf{R} \times \psi_\mathsf{B})(\psi_\mathsf{A}^{-1} \times \id_\mathsf{R}) (e_1\cdots e_m r).
\end{align*}
Consequently, $\varphi_\mathsf{R}^{(m)} =(\id_\mathsf{R} \times \psi_\mathsf{B})(\psi_\mathsf{A} \times \id_\mathsf{R})$.
\end{proof}

As a corollary, we get that aligned and balanced shift relations are actually equivalence relations. This follows readily from \cite[Proposition 4.4]{CDE} and Theorem \ref{thm:balanced->compatible}.

\begin{cor}
Let $\mathsf{A}$ and $\mathsf{B}$ be essential matrices indexed by $V$ and $W$ respectively, with entries in $\bbN$. Then aligned and balanced shift relations are equivalence relations.
\end{cor}

By \cite[Theorem~1.3]{CDE}, for essential matrices $\mathsf{A}$ and $\mathsf{B}$ with entries in $\bbN$, the matrices $\mathsf{A}$ and $\mathsf{B}$ are strong shift equivalent if and only if $\mathsf{A}$ and $\mathsf{B}$ are compatible shift equivalent.  Using this result and Theorem~\ref{thm:balanced->compatible}, we get the following corollary.

\begin{cor}\label{cor-matrices-equivalent-relations}
Let $\mathsf{A}$ and $\mathsf{B}$ be essential matrices with entries in $\bbN$, and let $m$ a positive integer. Then the following are equivalent.
\begin{enumerate}

\item $\mathsf{A}$ and $\mathsf{B}$ are balanced shift shift equivalent  with lag $m$;

\item $\mathsf{A}$ and $\mathsf{B}$ are aligned shift shift equivalent  with lag $m$; and

\item $\mathsf{A}$ and $\mathsf{B}$ are compatible shift equivalent with lag $m$.
\end{enumerate}
In particular, aligned and balanced shift equivalence relations coincide with SSE.
\end{cor}

\begin{rmk}
It is straightforward to show that SSE with lag $m$ implies aligned shift equivalence with lag $m$ (and therefore balanced or compatible shift equivalence with lag $m$). Although the converse is currently unknown, particularly since the proof of \cite[Theorem~1.3]{CDE} goes through C*-algebra (or through SFTs), it is likely that a bound on the lag of SSE can be obtained in terms of a lag of a concrete aligned, balanced or compatible shift equivalence.
\end{rmk}

\section{Strong shift equivalence for C*-correspondences}\label{sec-SSE-cor}

In this section, we study strong shift equivalence between C*-correspondences, which was introduced in a paper of Muhly, Pask and Tomforde \cite{MPT08}. The main result of this section is that when two regular (or both regular and full) C*-correspondences are strong shift equivalent, we may choose the C*-correspondences implementing a strong shift equivalence to also be regular (or both regular and full, respectively) C*-correspondences. This result will be important in Section~\ref{sec-modSSE-implies-SSE} where we prove that for finite essential matrices, strong shift equivalence of the matrices is equivalent to strong shift equivalent of their associated C*-correspondences.

\begin{dfn}\label{dfn-SSE-corr}
Let $X$ be an $\A$-correspondence and $Y$ be a $\B$-correspondence.  We say that \emph{$X$ is elementary strong shift related to $Y$} provided that there exist an $\A-\B$-correspondence $R$ and a $\B-\A$-correspondence $S$ such that 
$$
X \cong R \otimes_\B S \quad \text{and} \quad S \otimes_\A R \cong Y.
$$
We say that \emph{$X$ is strong shift equivalent to $Y$ with lag $m$} if there are correspondences $X_0, X_1, X_2, \ldots, X_m$ such that $X \cong X_0$, $X_m \cong Y$, and $X_i$ is elementary strong shift related to $X_{i+1}$ for all $i$, or equivalently, there exist $\C_i-\C_{i+1}$-correspondences $R_i$ and $\C_{i+1}-\C_i$-correspondences $S_i$ for $i=0,\dots, m-1$ such that $\C_0=\A$, $\C_m = \B$, $X \cong R_0 \otimes_{\C_1} S_0$, $S_m \otimes_{\C_{m-1}} R_m \cong Y$, and $R_i \otimes_{\C_{i+1}} S_i \cong S_{i+1} \otimes_{\C_{i+1}} R_{i+1}$.
\end{dfn}

\begin{rmk}
Let $X$ be a non-degenerate C*-correspondence over $\A$ and $Y$ a non-degenerate C*-correspondence over $\B$.  Suppose $X$ and $Y$ are SSE via possibly degenerate C*-correspon\-dences $R_1,..., R_m$ and $S_1, \ldots, S_m$, so that $S_i \otimes R_i \cong R_{i+1} \otimes S_{i+1}$ for $i=1,...,m-1$ and $X\cong R_1 \otimes S_1$, $Y \cong S_m \otimes R_m$.  By \cite[Lemma~3.10]{MPT08}, we may replace the $R_1,..., R_m$ and $S_1, \ldots, S_m$ with \emph{non-degenerate} correspondences $R_1',..., R_m'$ and $S_1', \ldots, S_m'$ implementing a SSE between $X$ and $Y$. Since $X$ and $Y$ are non-degenerate, we get that $X \cong R_1' \otimes S_1'$ and $Y\cong S_m' \otimes R_m'$, where the intermediary correspondences $S_i'\otimes R_i' \cong R_{i+1}' \otimes S_{i+1}'$ are now non-degenerate for $i=1,...,m-1$. Note that because of \cite[Lemma~3.10]{MPT08}, if any of the involved C*-correspondences above were \textit{a priori} assumed to be either injective, proper or full, this procedure will preserve their injectivity, properness or fullness, respectively.
\end{rmk}

In what follows, for more on quotients of C*-modules and C*-correspondences, we refer the reader to \cite{FMR-IUMJ2003, KatII}. Let $\I \unlhd \A$ be an ideal in a C*-algebra $\A$, and $X$ a right Hilbert C*-module over $\A$. By \cite[Corollary 1.4]{KatII} we get that $X\I$ is a closed right $\A$-linear subspace of $X$ which is invariant under the left action of $\L(X)$. Hence, we may consider $\K(X\I)$ as a subalgebra of $\K(X)$ which is the closed linear span of elements $\theta_{\xi,\eta}$ for $\xi,\eta \in X\I$.

We will denote by $X_\I$ the quotient Banach space $X/X\I$, which comes equipped with quotient maps $\A \rightarrow \A / \I$ and $X \rightarrow X_\I$ which are both denoted by $[\cdot]_{\I}$. The space $X_{\I}$ then carries a $\A/\I$ inner product and a right action of $\A/\I$ so that
$$
\langle[\xi]_{\I},[\eta]_{\I} \rangle = [\langle \xi,\eta \rangle]_{\I}, \ \ \text{and} \ \ [\xi]_{\I} \cdot [a]_{\I} = [\xi a]_{\I}
$$
for $\xi,\eta \in X$ and $a\in \A$. Together with this right $\A/ \I$ action and $\A/\I$-valued inner product, $X_{\I}$ turns into a Hilbert C*-module in its own right. Note that if $X$ is full then $X_{\I}$ is also full.

Since $X\I$ is invariant under the application of operators from $\L(X)$, we obtain a map $\L(X) \rightarrow \L(X_\I)$, also denoted by $[\cdot]_{\I}$ which satisfies $[S]_{\I}[\xi]_{\I} = [S\xi]_{\I}$ for $S\in \L(X)$ and $\xi \in X$. Then, by \cite[Lemma 1.6]{KatII} the restriction of $[\cdot]_{\I}$ to $\K(X)$ is surjective onto $\K(X_{\I})$ with kernel $\K(X\I)$ and is given by $[\theta_{\xi,\eta}]_{\I} = \theta_{[\xi]_{\I},[\eta]_{\I}}$.

Next, when we are given two ideals $\I,\I'$ such that $\I \subseteq \I'$, then $\I'/\I$ is an ideal of $\A / \I$, and the map $\big[ [a]_{\I} \big]_{\I'/\I} \rightarrow [a]_{\I'}$ gives a well-defined isomorphism between $(\A/\I)/ (\I'/\I)$ and $\A / \I'$, which we implicitly use to identify these C*-algebras henceforth. Similarly, we identify $(X_{\I})_{\I'/\I}$ with $X_{\I'}$. In all situations (at the level of the C*-algebras, Hilbert C*-modules and adjointable operators on Hilbert C*-modules) we have that the composition of $[\cdot]_{\I}$ with $[\cdot]_{\I'/\I}$ coincides with $[\cdot]_{\I'}$.

\begin{lem}[{\cite[Lemma~2.3]{FMR-IUMJ2003}}]\label{lem:quot-correspondence}
Let $X$ be right Hilbert $\B$-module and let $\phi \colon \A \to \L(X)$ be a $*$-homomorphism.  Suppose $\I \unlhd \A$ and $\J \unlhd \B$ such that 
$$\phi(\I)X \subseteq X\J.$$
Then $\phi_{\I, \J} \colon \A/\I \to \L(X_\J)$, given by $\phi_{\I,\J} ( [a]_{\I}) ([\xi]_{\J}) = [\phi(a)\xi]_{\J}$ for $a\in \A$ and $\xi \in X$, is a well-defined $*$-homomorphism. If moreover $\phi: \A \rightarrow \K(X)$, then $\phi_{\I,\J} : \A/\I \rightarrow \K(X_{\J})$.
\end{lem}

\begin{proof}
  Let $a \in \A$. Since $\phi(\I) X \subseteq X\J$, if $a \in \I$, then $[\phi(a) \xi]_{\J} = 0$ for all $\xi\in X$.  Therefore, $\I$ is in the kernel of the composition $\A\xrightarrow{\phi}\L(X)\xrightarrow{[\cdot]_\J} \L(X_\J)$. Consequently, this composition factors through a well-defined $*$-homomorphism $\phi_{\I, \J} \colon \A/\I \to \L(X_\J)$.
	
Finally, by the discussion preceding the statement of the Lemma, we see that if $\phi : \A \rightarrow \K(X)$, then also $\phi_{\I,\J} : \A / \J \rightarrow \K(X_{\J})$.
\end{proof}

By Lemma \ref{lem:quot-correspondence} we may henceforth consider $X_{\J}$ as a $\A/ \I - \B / \J$ correspondence. Note that if $X$ is an $\A$-correspondence and $\I = \ker(\phi_X)$, then 
$$
\phi_X(\I) X = \{0\} \subseteq X\I.
$$
Hence, by Lemma~\ref{lem:quot-correspondence}, $X_\I$ is an $\A/\I$-correspondence.

\begin{lem}\label{lem:compatible-condition}
Let $X$ be an $\A$-correspondence and let $Y$ be a $\B$-correspondence. Let $\I = \ker( \phi_X) \unlhd \A$ and $\J = \ker( \phi_Y) \unlhd \B$. 
 Suppose there are an $\A-\B$-correspondence $R$ and a $\B-\A$-correspondence $S$ such that 
$$X \cong R \otimes_{\B} S \quad \text{and} \quad S \otimes_{\A} R \cong Y.$$
Then 
$$
\phi_R(\I)R \subseteq R\J \quad \text{and} \quad \phi_S(\J) S \subseteq S\I.
$$
\end{lem}

\begin{proof}
Since $X \cong R \otimes_{\B} S$ we have that $\ker(\phi_X) = \ker(\phi_{R \otimes S} )$, where $\phi_{R\otimes S}(a) = \phi_{R}(a) \otimes 1$.  Thus, $\ker(\phi_R) \subseteq  \ker(\phi_{R \otimes S} ) = \ker(\phi_X)  =\I$. Similarly, we have that $\ker(\phi_Y) = \ker( \phi_{S\otimes R} )$ and $\ker(\phi_S ) \subseteq \J$.  

We now show $\phi_R(\I)R \subseteq R\J$. Let $a \in \I$.  Then $a \in \ker( \phi_{R \otimes S} )$.  Therefore, for all $r, r' \in R$ and for all $s, s' \in S$ we have
$$
0 = \langle \phi_R(a) r \otimes s , r' \otimes s' \rangle_{\A} = \langle s , \phi_S ( \langle \phi_R(a)r, r' \rangle_B ) s' \rangle_{\A}.
$$
Thus, $\phi_S ( \langle \phi_R(a)r, r' \rangle_{\B} ) = 0$ for all $r, r' \in R$ which implies $\langle \phi_R(a)r, r' \rangle_{\B}$ is an element of $\ker(\phi_S) \subseteq \J$ for all $r, r' \in R$.  Let $\{ e_\lambda \}$ be an approximate identity of $\J$ consisting of positive elements. Fix $r \in R$.  Then
\begin{align*}
&\langle \phi_R(a) r e_\lambda - \phi_R(a)r ,\phi_R(a) r e_\lambda - \phi_R(a)r \rangle_{\B} \\
&= e_\lambda \langle \phi_R(a) r, \phi_R(a) r \rangle_B e_\lambda - e_\lambda \langle \phi_R(a) r, \phi_R(a) r \rangle_{\B} \\
&\qquad \qquad -  \langle \phi_R(a) r, \phi_R(a) r \rangle_B e_\lambda + \langle \phi_R(a) r, \phi_R(a) r \rangle_{\B} \\
&\to 0.
\end{align*}
Since $\phi_R(a) r e_\lambda \in R\J$, the above implies $\phi_R(a) r \in R\J$.  Thus, we have proved that $\phi_R(\I)R \subseteq R\J$, and a similar argument shows $\phi_S(\J) S \subseteq S \I$.
\end{proof}

\begin{prp} \label{p:one-step-quotient}
Let $X$ be an $\A$-correspondence and let $Y$ be a $\B$-correspondence.  Let $\I \unlhd \A$ and $\J \unlhd \B$.  Suppose there is an $\A-\B$-correspondence $R$ and a $\B-\A$-correspondence $S$ such that 
$$X \cong R \otimes_{\B} S \quad \text{and} \quad S \otimes_{\A} R \cong Y$$
and 
$$
\phi_R(\I)R \subseteq R\J \quad \text{and} \quad \phi_S(\J) S \subseteq S\I.
$$
Then 
$$X_{\I} \cong R_{\J} \otimes_{\B/\J} S_{\I} \quad \text{and} \quad S_{\I} \otimes_{\A/\I} R_{\J} \cong Y_{\J}.$$
\end{prp}

\begin{proof}
Since $\phi_R(\I)R \subseteq R\J$ and $\phi_S(\J) S \subseteq S\I$, by Lemma~\ref{lem:quot-correspondence} we are able to construct the $\A/\I$-correspondence $X_{\I}$, the $\B/\J$-correspondence $Y_{\I}$, the $\A/\I-\B/\J$-correspondence $R_{\J}$, and the $\B/\J-\A/\I$-correspondence $S_{\I}$.  Let $U \colon X \to R \otimes_B S$ and $V \colon S \otimes_A R \to Y$ be unitary bimodule isomorphisms. Consider the map $\overline{U} \colon X_{\I} \to R_{\J} \otimes_{\B/\J} S_{\I}$ given by $\overline{U} ( [\xi]_{\I}) = ([\cdot]_{\J} \otimes [\cdot]_{\I}) \circ U(\xi)$. Since 
\begin{align*}
([\cdot]_{\J} \otimes [\cdot]_{\I}) \circ U(\xi a) &=([\cdot]_{\J} \otimes [\cdot]_{\I}) ( U(\xi) a) = ([\cdot]_{\J} \otimes [\cdot]_{\I})( U(\xi)) [a]_{\I}\\
&= ([\cdot]_{\J} \otimes [\cdot]_{\I})( U(x)) \cdot 0=0,
\end{align*}
for all $\xi \in X$ and for all $a \in \I$, we get that $\overline{U}$ is well-defined. It is clear that $\overline{U}$ is a linear map and respects the right $\A/\I$-module structure. Since
\begin{align*}
[a]_{\I} \cdot ( [r]_{\J} \otimes [s]_{\I}) &= ( [a]_{\I} \cdot [r]_{\J} ) \otimes [s]_{\I}= [\phi_R(a)r]_{\J} \otimes [s]_{\I} = ([\cdot]_{\J} \otimes [\cdot]_{\I})( \phi_R(a) r \otimes s) \\
&= ([\cdot]_{\J} \otimes [\cdot]_{\I})( a \cdot ( r \otimes s) ). 
\end{align*}
for all $a \in \A$, for all $r \in R$ and $s \in S$, $[\cdot]_{\J} \otimes [\cdot]_{\I}$ is a left $\A/\I$-module map. Hence,
\[
[a]_{\I} \cdot \overline{U}( [\xi]_{\I}) = [a]_{\I} \cdot ([\cdot]_{\J} \otimes [\cdot]_{\I}) \circ U(\xi) = ([\cdot]_{\J} \otimes [\cdot]_{\I}) ( a \cdot U(\xi)) = ([\cdot]_{\J} \otimes [\cdot]_{\I}) ( U( a\cdot \xi))
\]
for all $a \in \A$ and for all $\xi \in X$.  Thus, $\overline{U}$ is a left $\A/\I$-module map. Since $U$ and $[\cdot]_{\J} \otimes [\cdot]_{\I}$ are surjections, $\overline{U}$ is also a surjection.

We are left with showing that $\langle \overline{U} ([\xi]_{\I}), \overline{U}( [\eta]_{\I}) \rangle_{ \A/ \I } = \langle [\xi]_{\I}, [\eta]_{\I} \rangle_{\A/\I}$ for all $\xi,\eta \in X$.  Since 
\begin{align*}
&\langle ([\cdot]_{\J} \otimes [\cdot]_{\I}) ( r \otimes s ) , ([\cdot]_{\J} \otimes [\cdot]_{\I}) (r' \otimes s') \rangle_{\A/\I}\\
&= \langle [s]_{\I}, \langle [r]_{\J}, [r']_{\J} \rangle_{\B/\J} \cdot [s']_{\I} \rangle_{\A/\I} \\
            &= \langle [s]_{\I}, [\langle r, r'\rangle_B]_{\J} \cdot [s']_{\I} \rangle_{\A/\I} \\
            &= \langle [s]_{\I}, [ \phi_S ( \langle r, r' \rangle_{\B} )s' ]_{\I} \rangle_{\A/\I} \\
            &=[ \langle s , \phi_S ( \langle r, r' \rangle_B )s' \rangle_{\A} ]_{\I} \\
            &= [ \langle r \otimes s , r' \otimes s' \rangle_A ]_{\I},
\end{align*}
for all $r,r' \in R$ and for all $s, s' \in S$ we get that $\langle ([\cdot]_{\J} \otimes [\cdot]_{\I}) (w ) , ([\cdot]_{\J} \otimes [\cdot]_{\I}) (z) \rangle_{\A/\I} = [ \langle w, z \rangle_A ]_{\I}$ for all $w, z \in R \otimes_{\B} S$.  Therefore,
\begin{align*}
  \langle \overline{U} ([\xi]_{\I}), \overline{U}( [\eta]_{\I}) \rangle_{ \A/ \I } &= \langle ([\cdot]_{\J} \otimes [\cdot]_{\I}) (U(\xi)) , ([\cdot]_{\J} \otimes [\cdot]_{\I}) (U(\eta)) \rangle_{\A/\I}  \\
    &= [ \langle U(\xi), U(\eta) \rangle_{\A} ]_{\I} \\
    &= [ \langle \xi, \eta \rangle_{\A} ]_{\I} \\
    &= \langle [\xi]_{\I}, [\eta]_{\I} \rangle_{\A/\I}
\end{align*}
for all $\xi,\eta \in X$.  Thus, completing the proof of the claim.

We conclude that $\overline{U} \colon X_{\I} \to R_{\J} \otimes_{\B/\J} S_{\I}$ is a unitary correspondence isomorphism.  A similar argument show that $\overline{V} \colon Y_{\J} \to S_{\I} \otimes_{\A/\I} R_{\J}$ is a unitary correspondence isomorphism.  Thus we proved that
$$X_{\I} \cong R_{\J} \otimes_{\B/\J} S_{\I} \quad \text{and} \quad S_{\I} \otimes_{\A/\I} R_{\J} \cong Y_{\J},$$
as required.
\end{proof}

\begin{exm}
If the correspondence $X$ is a proper $\A$-correspondence, then $X_{\I}$ is not necessarily regular when $\I = \ker(\phi_X)$. For example, let $E$ be the graph 
$$
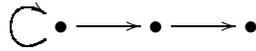
\xymatrix{
\bullet \ar@(dl, ul)[] \ar[r] & \bullet \ar[r] &\bullet
}
$$

\bigskip
Let $X = X(E)$ be the associated graph correspondence and $\I = \bbC \delta_v$, where $v$ is the only sink.  By \cite[Example~2.4]{FMR-IUMJ2003}, $X_{\I}$ is isomorphic to the graph correspondence of the following graph 
$$
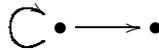
\xymatrix{
\bullet \ar@(dl, ul)[] \ar[r] & \bullet 
}
$$
\bigskip
which is not regular.  Of course, we can repeat the process again to get a regular correspondence isomorphic to the graph correspondence of 
$$
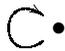
\xymatrix{
\bullet \ar@(dl, ul)[]  
}
$$
Hence, in general we will need to repeat this process of taking quotients by kernels of left actions.
\end{exm}

The following addresses the issue above, by producing from any pair $X$ and $Y$ of elementary strong shift related correspondences, a new pair $X_{\I}$ and $Y_{\J}$ of \emph{injective} correspondences which are still elementary strong shift related.  Moreover, if $X$ is an injective correspondence, then the procedure gives $\I = \{0 \}$ and hence $X \cong X_\I$.  Similarly, if $Y$ is an injective correspondence, then $Y_{\J} \cong Y$.

Let $Z$ be an $\A-\B$-correspondence and $\I\unlhd \B$ be a closed ideal.
Consider the ideal
\[
Z^{-1}(\I) = \{a\in \A\colon aZ\subseteq Z\I\} = \{ a \in \A \colon \langle z', \phi_Z(a) z \rangle \in \I , \text{ for all } z, z' \in Z \}
\]
of $\A$ from \cite[Definition 4.1]{KatII}.
It is the biggest ideal of $\A$ such that $Z^{-1}(\I)Z\subseteq Z\I$.
Consequently, $Z_{\I}$ can be endowed with the structure of an injective $\A/Z^{-1}({\I})-B/\I$-correspondence.
We need two important lemmas about this construction.
\begin{lem}\label{l:inv-intersection}
  Let $\{\I_j\}_{j\in \mathfrak J}$ be a collection of ideals of $\B$.
  Then,
  \[
    Z^{-1}(\bigcap_{j\in\mathfrak{J}} \I_j) = \bigcap_{j\in \mathfrak{J}} Z^{-1}(\I_j).
  \]
\end{lem}
\begin{proof}
  Let $\I = \bigcap_{j\in \mathfrak{J}} \I_j$.
  Since $Z\I = Z\bigcap_{j\in \mathfrak{J}} \I_j = \bigcap_{j\in \mathfrak{J}}Z \I_j$, we have $aZ\subseteq Z\I$ if and only if $aZ\subseteq Z\I_j$ for all $j\in \mathfrak{J}$.
  The statement then follows by definition of $Z^{-1}$ operation.
\end{proof}

\begin{lem}\label{l:inv-composition}
  Let $Z_1$ be an $\A-\B$-correspondence and $Z_2$ be a $\B-\mathcal{C}$-correspondence.
  Then, for any ideal $\I\unlhd \mathcal{C}$ we have
  \[
    (Z_1\otimes_{\B} Z_2)^{-1}(\I) = Z_1^{-1}(Z_2^{-1}(\I)).
  \]
\end{lem}
\begin{proof}
  Let $a\in \A$ be an arbitrary element.
  By \cite[Proposition 1.3]{KatII}, we have $a(Z_1 \otimes Z_2)\subseteq (Z_1\otimes Z_2)\I$ if and only if $\langle \xi', \phi_{Z_1\otimes Z_2}(a) \xi\rangle \in \I$ for all $\xi, \xi'\in Z_1\otimes Z_2$.
  By linearity and continuity, it is enough to show the inclusion only for elementary tensors $\xi = \xi_1\otimes \xi_2$, $\xi'=\xi_1'\otimes \xi_2'$ .
  We thus have that $a\in (Z_1\otimes Z_2)^{-1}(\I)$ if and only if
  \[
    \langle \xi_1'\otimes \xi_2', \phi_{Z_1\otimes Z_2}(a)\xi_1\otimes \xi_2\rangle = \langle \xi_2', \phi_{Z_2}\left(\langle \xi_1', \phi_{Z_1}(a)\xi_1\rangle\right) \xi_2\rangle \in \I
  \]
  for all $\xi_i,\xi_i'\in Z_i$.
  Again, by applying \cite[Proposition 1.3]{KatII} twice, the latter is equivalent to $a\in Z_1^{-1}(Z_2^{-1}(\I))$.
 \end{proof} 

 For an $\A$-correspondence $X$, we say that an ideal $\I\unlhd \A$ is \emph{fully $X$-invariant} if $\I = X^{-1}(\I)$.
The name comes from the fact that it is a strengthening of Katsura's notion of $X$-invariance \cite[Definition 4.8]{KatII}.
The ideal $\A\unlhd \A$ is fully $X$-invariant, so the set of ideals with this property is non-empty.
Moreover, by Lemma~\ref{l:inv-intersection} any intersection of fully $X$-invariant ideals is fully $X$-invariant.
Therefore, there is the smallest fully $X$-invariant ideal which we denote by $\I_X$.
We have $X^{-1}(0) = \ker \phi_X$, so $0$ is fully $X$-invariant if and only if $X$ is injective.
Hence, as $0$ is the smallest ideal of $\A$, we have that $\I_X = 0$ if and only if $X$ is injective.

\begin{prp} \label{prp:to-injective}
Let $X$ be an $\A$-correspondence and $Y$ a $\B$-correspondence. Suppose there exist an $\A-\B$-correspondence $R$ and a $\B-\A$ correspondence $S$ such that
$$
X \cong R\otimes_{\B} S \ \ \text{and} \ \ S\otimes_{\A}R \cong Y.
$$
Then $X_{\I_X}, Y_{\I_X}$ are injective C*-correspondences over $\A/\I_X$ and $\B/\I_Y$ respectively, satisfying
$$
X_{\I_X} \cong R_{\I_Y}\otimes_{\B / \I_Y} S_{\I_X} \ \ \text{and} \ \ S_{\I_X}\otimes_{\A / \I_X}R_{\I_Y} \cong Y_{\I_Y}.
$$
Consequently, $R_{\I_Y}$ and $S_{\I_X}$ are also injective C*-correspondences.
\end{prp}
\begin{proof}
  We can apply Lemma \ref{l:inv-composition} to $X = R\otimes_{\B} S$ to get
  \[
    X^{-1}(R^{-1}(\I_Y)) \cong R^{-1}(S^{-1}(R^{-1}(\I_Y)))=R^{-1}(Y^{-1}(\I_Y)) = R^{-1}(\I_Y),
  \]
  where the last equality follows from the full $Y$-invariance of $\I_Y$.
  Therefore, $R^{-1}(\I_Y)$ is also fully $X$-invariant and, hence, $\I_X\subseteq R^{-1}(\I_Y)$ by definition of $\I_X$.
  This proves that $\phi_R(\I_X)R\subseteq R\I_Y$ and the inclusion $\phi_S(\I_Y)S \subseteq S\I_X$ can be obtained analogously.

  We can now apply Proposition \ref{p:one-step-quotient} to obtain
  \[
    X_{\I_X} \cong R_{\I_Y}\otimes_{\B / \I_Y} S_{\I_X} \ \ \text{and} \ \ S_{\I_X}\otimes_{\A / \I_X}R_{\I_Y} \cong Y_{\I_Y}.
\]
If $[a]_{\I_X}\in \A/\I_X$ is in the kernel of $\phi_{X_{\I_X}}$, then $\phi_X(a)X\subset X\I_X$, which implies $a\in \I_X$ or $[a]_{\I_X}=0$ since $\I_X$ is fully $X$-invariant.
Therefore, $X_{\I_X}$ is an injective $\A/\I_X$-correspondence and, similarly, $Y_{\I_Y}$ is an injective $\B/\I_Y$-correspondence.
\end{proof}

Our next order of business is to cut down an elementary shift relation between injective $X$ and $Y$ so as to obtain an elementary shift relations between full or regular C*-correspondences.

\begin{lem} \label{lem:invariant-full-ideal}
Let $X$ be an $\A$-correspondence and let $Y$ be a $\B$-corres\-pondence. Suppose $R$ is an $\A-\B$-correspondence and $S$ is a $\B-\A$-correspondence such that 
$$
X \cong R \otimes_{\B} S \quad \text{and} \quad Y \cong S \otimes_{\A} R.
$$
Let $\I = \langle X, X \rangle$ and $\J = \langle Y,Y \rangle$.  Then 
$$
\I R \subseteq R \J \quad \text{and} \quad \J S \subseteq S \I
$$
\end{lem}

\begin{proof}
We prove the inclusion $\I R \subseteq R\J$, where the inclusion $\J S \subseteq S \I$ is be proven similarly.

Note first that $\langle R \otimes S, R\otimes S \rangle = \I$ because $R\otimes S \cong X$. Now, for $r_1,r_2 \in R$, $s_1,s_2 \in S$ elements of the form $\langle r_1 \otimes s_1 , r_2 \otimes s_2 \rangle$ generate the ideal $\langle R \otimes S, R\otimes S \rangle$, and for $r_1',r_2' \in R$ since
$$
\langle \phi_R(\langle r_1 \otimes s_1 , r_2 \otimes s_2 \rangle) r_1',r_2' \rangle = \langle \phi_S(\langle s_1, \phi_R(\langle r_1,r_2 \rangle) s_2 \rangle) r_1',r_2' \rangle = 
$$
$$
\langle r_1', \phi_S(\langle s_2, \phi_R(\langle r_2,r_1 \rangle) s_1 \rangle) r_2' \rangle = \langle s_2 \otimes r_1', \phi_R(\langle r_2,r_1 \rangle) s_1 \otimes r_2' \rangle,
$$
we see that $\langle \I R, \I R \rangle \subseteq \langle S \otimes R, S \otimes R \rangle = \J$, where the latter equality holds because $S\otimes R \cong Y$. Thus, since $\I R = \I R \langle \I R, \I R \rangle$, we get that $\I R = \I R \langle \I R, \I R \rangle \subseteq \I R \J \subseteq R \J$. This concludes the proof.
\end{proof}

The following is essentially from \cite[Lemma~4.2]{FMR-IUMJ2003}, where the assumption of injectivity of $\phi$ can automatically be dispensed with.

\begin{lem}\label{lem:gen-fmr}
Let $X$ be a Hilbert $\A$-module and let $Y$ be a Hilbert $\B$-module.  Suppose $\phi \colon \A \to \L(Y)$ is a $*$-homomorphism and suppose $V \in \L(X)$ such that $V \otimes 1 \in \K( X \otimes_{\A} Y)$.  Then $\ran(V) \subseteq X\phi^{-1}(\K(Y))$.
\end{lem}

\begin{proof}
For each $\xi \in X$, define $T_{\xi} \colon Y \to X \otimes_{\A} Y$ by $T_{\xi}(\eta) = \xi \otimes \eta$ for all $\eta \in Y$.  First note that for all $\xi, \xi' \in X$
\begin{align*}
    T_{\xi}^* (V \otimes 1 ) T_{\xi'} (\eta) &= T_{\xi}^* (V \otimes 1)(\xi' \otimes \eta) \\
    &= T_{\xi}^* ( V\xi' \otimes \eta) \\
    &= \phi ( \langle \xi, V\xi' \rangle ) \eta.
\end{align*}
Therefore, $T_{\xi}^* (V \otimes 1 ) T_{\xi'}= \phi ( \langle \xi, V\xi' \rangle )$.  

We first claim that $T_{\xi}^* \K( X \otimes_{\A} Y) T_{\xi'} \subseteq \K(Y)$ for all $\xi, \xi' \in X$.  Let $\xi, \xi' \in X$.  Then
\begin{align*}
    T_{\xi}^* \theta_{ \xi_1 \otimes \eta_1, \xi_2 \otimes \eta_2} T_{\xi'} ( \eta) &= T_{\xi}^* ( \xi_1 \otimes \eta_1 \langle \xi_2 \otimes \eta_2 , \xi' \otimes \eta \rangle) \\
    &= (\phi(\langle \xi, \xi_1 \rangle)\eta_1 ) \langle \eta_2, \phi(\langle \xi_2, \xi' \rangle) \eta \rangle \\
    &= (\phi(\langle \xi, \xi_1 \rangle )\eta_1 ) \langle \phi ( \langle \xi' , \xi_2 \rangle ) \eta_2, \eta \rangle \\
    &= \theta_{\phi(\langle \xi, \xi_1 \rangle )\eta_1 , \phi ( \langle \xi' , \xi_2 \rangle_A ) \eta_2 }(\eta)
\end{align*}
for all $\eta \in Y$.  Hence, $T_{\xi}^* \theta_{ \xi_1 \otimes \eta_1, \xi_2 \otimes \eta_2} T_{\xi'}=\theta_{\phi(\langle \xi, \xi_1 \rangle )\eta_1 , \phi ( \langle \xi' , \xi_2 \rangle ) \eta_2 } \in \K(Y)$.  This proves our claim.

Since $V \otimes 1 \in \K(X \otimes_\A Y)$, we get from our claim that
$$
\phi ( \langle \xi, V\xi' \rangle )=T_{\xi}^* (V \otimes 1 ) T_{\xi'} \in \K(Y)
$$
for all $\xi,\xi' \in X$. Thus, $\langle \xi , V\xi' \rangle \in \phi^{-1}(\K(Y))$ which implies that for all $\xi' \in X$, $V\xi' \in X \phi^{-1}(\K(Y))$.
\end{proof}

\begin{lem} \label{l:invariant-ideal}
Let $X$ be an injective $\A$-correspondence and let $Y$ be an injective $\B$-correspondence. Suppose $R$ is an $\A-\B$-correspondence and $S$ is a $\B-\A$-correspondence such that 
$$
X \cong R \otimes_{\B} S \quad \text{and} \quad Y \cong S \otimes_{\A} R.
$$
Let $\I = \phi_X^{-1}(\K(X))$, $\I' = \phi_R^{-1}(\K(R))$, $\J = \phi_Y^{-1}(\K(Y))$, and $\J' = \phi_S^{-1} (\K(S))$.  Then 
$$
\I R = R\J' \quad \text{and} \quad \J S = S\I'
$$
\end{lem}

\begin{proof}
For any $a \in \I$ we have $\phi_{R} (a) \otimes 1 \in \K ( R \otimes_{\B} S)$ since $\phi_X( a) \in \K(X)$ and $X \cong R \otimes_{\B} S$.  By Lemma~\ref{lem:gen-fmr}, $\phi_R(a)R \subseteq R\J'$.  Thus, $\I R \subseteq R\J'$.  

Let $r_1, r_2, r \in R$, let $s \in S$, and let $b \in \J'$.  Then
\begin{align*}
    T_{r_1} \phi_S(b) T_{r_2}^* ( r \otimes s) &= T_{r_1} \phi_S(b) ( \phi_S ( \langle r_2, r \rangle ) s ) \\
            &= r_1 \otimes \phi_S(b) \phi_S ( \langle r_2, r \rangle ) s \\
            &= r_1 \otimes \phi_S ( b \langle r_2, r\rangle) s \\
            &= r_1 b \langle r_2, r\rangle \otimes s \\
            &= \theta_{ r_1 b , r_2 }(r) \otimes s \\
            &= (\theta_{r_1 b , r_2 } \otimes 1) ( r \otimes s ).
\end{align*}
Thus, $T_{r_1} \phi_S(b) T_{r_2}^*= \theta_{r_1 b , r_2 } \otimes 1$ for all $r_1, r_2 \in R$.  Since $b \in \J' = \phi_S^{-1} ( \K(S))$ we have that $\phi_S(b) \in \K(S)$, and as straightforward computation on rank one operators shows, we get that
$$
\theta_{r_1 b , r_2 }\otimes 1= T_{r_1} \phi_S(b) T_{r_2}^* \in \K( R \otimes_B S).
$$

Since $\I = \phi_X^{-1} ( \K(X)) = \phi_{R \otimes_B S}^{-1}( \K( R \otimes_B S))$, by the computation above we get that there exists $a \in \I$ such that
$$
\theta_{r_1 b , r_2 }\otimes 1 = \phi_R(a) \otimes 1.
$$
Since $\phi_X$ and $\phi_Y$ are injective, $\phi_R$ and $\phi_S$ are injective. Hence, by \cite[Lemma~4.2]{FMR-IUMJ2003}, $\phi_R(a) = \theta_{r_1 b , r_2 }$.  In particular, for all $r \in R$,
$$
r_1b \langle r_2, r \rangle = \phi_R(a) r \in \I R.
$$
Therefore, as $b\in \J'$ is arbitrary, we have
$$
R\J' \langle R\J', R\J' \rangle_{\J'} \subseteq \I R.
$$
Since $R\J'$ is a $\J'$-Hilbert module, $R\J' \langle R\J' , R\J' \rangle_{\J'}$ is dense in $R\J'$.  Therefore, $R\J' \subseteq \I R$.  We have just shown $\I R \subseteq R\J' \subseteq \I R$, thus, $\I R=R\J'$.  The symmetric argument shows that $\J S = S\I'$.
\end{proof}

Let $X$ be a $\A$-correspondence and $Y$ a $\B$-correspondence. Suppose $R$ is an $\A-\B$-correspondence and suppose $S$ is a $\B-\A$-correspondence such that $X \cong R \otimes_{\B} S$. Let $\I \lhd \A$ and $\J \lhd \B$ be two ideals. Then there is always an isometric bimodule map $\overline{\psi}: \I R \J \otimes_{\J} S \I \rightarrow \I R \otimes_{\B} S \I \cong \I X \I$. Indeed, define $\psi \colon \I R \J \times \J S \I \to \I R \otimes_{\B} S \I$ by $\psi( r, s) = r \otimes_{\B} s$.  It is clear that $\psi$ is bilinear and balanced since 
$$
\psi( rb, s) = rb \otimes_{B} s = r \otimes_{B} bs = \psi( r, b s)
$$
for all $r \in \I R \J$, all $s\in \J S \I$, and all $b \in \J$. Thus, $\psi$ induces a module map from the balanced algebraic tensor product $\overline{\psi} \colon \I R \J \odot_{\J} \J S \I \to \I R \otimes_{\B} S \I$. Since 
\begin{align*}
\langle \overline{\psi} ( r_1 \otimes_{\J} b_1 s_1 ), \overline{\psi} ( r_2 \otimes_{\J} b_2 s_2 ) \rangle_{\I} &= \langle \overline{\psi} ( r_1 b_1 \otimes_{\J} s_1 ), \overline{\psi} ( r_2 b_2 \otimes_{\J} s_2 ) \rangle_{\I}  \\
&= \langle r_1 b_1 \otimes_{\B} s_1 , r_2 b_2 \otimes_{\B} s_2 \rangle_{\I} \\
&= \langle s_1 , \phi_S (\langle r_1 b_1 , r_2 b_2 \rangle_B ) s_2 \rangle_{\A} \\
&= \langle r_1 b_1 \otimes_{\J} s_1, r_2 b_2 \otimes_{\J} s_2 \rangle_{\I} \\
&= \langle r_1 \otimes_{\J} b_1 s_1 , r_2 \otimes_{\J} b_2 s_2 \rangle_{\I}
\end{align*}
for all $r_1, r_2 \in R \J$, for all $s_1, s_2 \in S \I$, and for all $b_1, b_2 \in \J$. Thus, we see that $\overline{\psi}$ can be extended to a bimodule map which we continue to denote by $\overline{\psi}$ from $\I R \J \otimes_{\J} \J S \I$ to $\I R \otimes_{B} S \I$.  Moreover, the above computation also implies 
$$
\langle \overline{\psi}(\xi), \overline{\psi}(\eta) \rangle_{\I} = \langle \xi , \eta \rangle_{\I}
$$
for all $\xi,\eta \in \I R \J \otimes_{\J} \J S\I$, so that $\overline{\psi}$ is a unitary correspondence isomorphism onto its image.

\begin{prp} \label{prp:to-full}
Let $X$ be a $\A$-correspondence and let $Y$ be a $\B$-correspondence.  Suppose $R$ is an $\A-\B$-correspondence and suppose $S$ is a $\B-\A$-correspondence such that 
$$
X \cong R \otimes_{\B} S \quad \text{and} \quad Y \cong S \otimes_{\A} R.
$$
Let $\I = \langle X,X \rangle$ and $\J = \langle Y, Y \rangle$.  Then  
$$
\I R \J \otimes_{\J} \J S \I \cong \I X \I \quad \text{and} \quad \J S \I \otimes_{\I} \I R \J \cong \J Y \J.$$
Moreover, if $X$ is full then $ \J Y \J$ is also full.
\end{prp}

\begin{proof}
We claim $\I R \J \otimes_{\J} S \I \cong \I R \otimes_{\B} S \I$, where a similar argument would show that $\J S \I \otimes_{\I} \I R \J \cong \J R \otimes_{\A} S \J$. Since $\I R \otimes_{\B} S \I \cong \I X \I$ and $ \J R \otimes_{\A} S \J \cong \J Y \J$, this would conclude the proof.

As we have seen in the discussion preceding the proposition, there is an isometric bimodule map $\overline{\psi} : \I R \J \otimes_{\J} S \I \rightarrow \I R \otimes_{\B} S \I \cong \I X \I$ given by $\overline{\psi}(r\otimes_{\J} s) = r\otimes_{\B} s$. Thus, to finish the proof we must show that $\overline{\psi}$ is surjective.

Take $\{ e_\lambda \}$ to be an approximate identity $\{e_\lambda\}$ of positive elements for $\J$. Then, By Lemma~\ref{lem:invariant-full-ideal} we have that $\J S \subseteq S \I$ and $\I R \subseteq R \J$. Thus, for any $r\in \I R$ we have that
$$
\lim_{\lambda} r e_{\lambda} = r.
$$
Hence,
$$
r\otimes_{\B} s = \lim_{\lambda} r e_{\lambda} \otimes_{\B} s = \lim_{\lambda} r \sqrt{e_{\lambda}} \otimes_{\B} \sqrt{e_{\lambda}}s = \lim_{\lambda} \overline{\psi} (r \sqrt{e_{\lambda}} \otimes_{\J} \sqrt{e_{\lambda}}s).
$$
Therefore, we get that $\overline{\psi}$ is a surjection. Thus, we have just shown that $\overline{\psi} \colon \I R \J \otimes_{\J} \J S \I \to \I R \otimes_{\B} S \I$ is a unitary correspondence isomorphism. This concludes the proof.

Suppose now that $X$ is full, so that $X \cong R\mathcal{J} \otimes_{\mathcal{J}} \mathcal{J}S$. Then,
$$
\langle X, X \rangle = \langle \mathcal{J} S, \phi_R(\langle R \mathcal{J}, R \mathcal{J}) \rangle \mathcal{J} S \rangle \subseteq \langle \mathcal{J} S, \mathcal{J} S \rangle.
$$
Hence, we see that $\J S$ is full. Since $R$ is non-degenerate we get that
$$
\langle \J Y \J, \J Y \J \rangle = \langle R \J, \phi_S(\langle \J S, \J S\rangle) R \J \rangle = \langle R \J, R \J \rangle,
$$
but now since
$$
 \langle Y \J, Y \J \rangle = \langle R \J, \phi_S(\langle S, S \rangle) R \J \rangle \subseteq \langle R \J, R \J \rangle \subseteq \langle \J Y \J, \J Y \J \rangle,
$$
and since $Y \J$ is full by definition, we get that $\J Y \J$ is full as well.
\end{proof}

Note that if in Proposition \ref{prp:to-full} $X$ (or $Y$) was full, then $X\cong \I X \I$ (or $Y \cong \J Y \J$ respectively). Moreover, if $X$ (or $Y$) were regular to begin with, then $\I X\I$ (or $\J Y \J$ respectively) would be regular as well.

\begin{prp} \label{prp:to-regular}
Let $X$ be an injective $\A$-correspondence and let $Y$ be an injective $\B$-correspondence.  Suppose $R$ is an $\A-\B$-correspondence and suppose $S$ is a $\B-\A$-correspondence such that 
$$
X \cong R \otimes_{\B} S \quad \text{and} \quad Y \cong S \otimes_{\A} R.
$$
Let $\I = \phi_X^{-1}(\K(X))$ and $\J = \phi_Y^{-1}(\K(Y))$.  Then  
$$
\I R \J \otimes_{\J} \J S \I \cong \I X \I \quad \text{and} \quad \J S \I \otimes_{\I} \I R \J \cong \J Y \J.$$
\end{prp}

\begin{proof}
We claim $\I R \J \otimes_{\J} S \I \cong \I R \otimes_{\B} S \I$, where a similar argument would show that $\J S \I \otimes_{\I} \I R \J \cong \J R \otimes_{\A} S \J$. Since $\I R \otimes_{\B} S \I \cong \I X \I$ and $ \J R \otimes_{\A} S \J \cong \J Y \J$, this would conclude the proof.

As we have seen in the discussion preceding Proposition \ref{prp:to-full} there is an isometric bimodule map $\overline{\psi} : \I R \J \otimes_{\J} S \I \rightarrow \I R \otimes_{\B} S \I \cong \I X \I$ given by $\overline{\psi}(r\otimes_{\J} s) = r\otimes_{\B} s$. To finish the proof, we must show that $\overline{\psi}$ is a surjection.

Let $r \in \I R$ and $s \in S \I$.  Set $\I' = \phi_R^{-1}(\K(R))$ and $\J' = \phi_S^{-1} (\K(S))$.  Since $\phi_X$ and $\phi_Y$ are injections and since
$$
\ker( \phi_R ) \subseteq \ker( \phi_X) \quad \text{and} \quad \ker(\phi_S) \subseteq \ker(\phi_X),
$$
we get that $\phi_R$ and $\phi_S$ are injections. Therefore, by \cite[Lemma~4.2]{FMR-IUMJ2003}, $\I \subseteq \I'$ and $\J \subseteq \J'$. By Lemma~\ref{l:invariant-ideal}, $\I R = R \J'$ and $\J S = S \I'$.  Therefore, $S \I \subseteq S \I' = \J S$ and $R \J \subseteq R \J' = \I R$.  Take $\{ e_\lambda \}$ to be an approximate identity $\{e_\lambda\}$ of positive elements for $\J$.  Since $S\I \subseteq \J S$,
$$
\lim_\lambda e_\lambda s = s \ \ \text{for} \ \ s\in S \I.
$$
Then
$$
r\otimes_{\B} s = \lim_{\lambda } r \otimes_{\B} e_\lambda s = \lim_{ \lambda } r \sqrt{e_{\lambda}} \otimes_{\B} \sqrt{e_{\lambda }}s = \lim_{ \lambda }\overline{\psi} ( r \sqrt{e_\lambda} \otimes_{\J}  \sqrt{e_\lambda} s ).
$$
Therefore, $\overline{\psi}$ is a surjection. Thus, we have just shown that $\overline{\psi} \colon \I R \J \otimes_{\J} \J S \I \to \I R \otimes_{\B} S \I$ is a unitary correspondence isomorphism. This concludes the proof.
\end{proof}

\begin{thm} \label{thm:reg-reduction}
Let $X$ be an $\A$-correspondence and $Y$ be a $\B$-correspondence. Suppose that both $X$ and $Y$ are regular (or both regular and full) and that there exist C*-correspondences $R_1,...,R_m$ and $S_1,...,S_m$ which implement a strong shift equivalence between $X$ and $Y$ with lag $m$. Then there exist C*-correspondences $R_1',..., R_m'$ and $S_1',...,S_m'$, such that the pairs $R_i' \otimes S_i'$ and $S_{i+1}' \otimes R_{i+1}'$ are regular (or both regular and full, respectively), which implement a strong shift equivalence between $X$ and $Y$ with lag $m$.
\end{thm}

\begin{proof}
By assumption, we have that
$$
X \cong R_1 \otimes S_1, \ \ S_m \otimes R_m \cong Y
$$
and $X_i= S_i \otimes R_i \cong R_{i+1} \otimes S_{i+1}$ for all $i=1,...,m-1$.

By Proposition \ref{prp:to-injective}, $X = X_{\I_X}$ is elementary strong shift related to the injective correspondence $X'_1 = (X_1)_{\I_{X_1}}$ via $R_1'=(R_1)_{\I_{X_1}}$ and $S_1' = (S_1)_{\I_{X}}$.
Analogously, if we define $X_{i}'= (X_{i})_{\I_{X_i}}$, $R_i':=(R_i)_{\I_{X_i}}$ and $S_i':= (S_i)_{\I_{X_{i+1}}}$, we get that $X_i'$ is elementary strong shift related to $X_{i+1}'$ for $i=1,\dots, m-2$ via $R_i',S_i'$ and $X_{m-1}'$ is elementary strong shift related to $Y=Y_{\I_Y}$ via $R_m'$, $S_m'$.
Since $X$, $Y$, and $X_i'$ are injective, $R_i'$ and $S_i'$ are also injective.
Therefore, we constructed a strong shift equivalence between $X$ and $Y$ implemented by injective C*-correspondences.

Hence, we may assume without loss of generality that $R_1,...,R_m$ and $S_1,...,S_m$ in the statement of the theorem are injective to begin with. In this case we apply Proposition \ref{prp:to-regular} to get new C*-correspondences $R_1',...R_m'$ and $S_1',...,S_m'$ such that the pairs $R_i' \otimes S_i' \cong S_{i+1}' \otimes R_{i+1}'$ are \emph{proper} for $i=1,...,m-1$. But now, since by assumption $X$ and $Y$ are proper, we get that $X \cong R_1'\otimes S_1'$ and $S_m' \otimes R_m' \cong Y$. We can now repeat the process of injectivization using Proposition \ref{prp:to-injective} to obtain C*-correspondences $R_1'',...R_m''$ and $S_1'',...,S_m''$ so that by Lemma \ref{lem:quot-correspondence} the C*-correspondences $R_i'' \otimes S_i''\cong S_{i+1}'' \otimes R_{i+1}''$ are still proper, and injective by design.

Finally, if $X$ and $Y$ are both regular and full to begin with, and $R_1'',...R_m''$ and $S_1'',...,S_m''$ are the C*-correspondences from the previous paragraph, we may apply Proposition \ref{prp:to-full} to obtain new C*-correspondences $R_1''',...R_m'''$ and $S_1''',...,S_m'''$ so that $R_i''' \otimes S_i''' \cong S_{i+1}''' \otimes R_{i+1}'''$ are still proper. By induction, that $X$ and $Y$ are full and the last part of Proposition \ref{prp:to-full}, we see that $R_i'''' \otimes S_i''' \cong S_{i+1}''' \otimes R_{i+1}'''$ are also full. Finally, we can now repeat once more the process of injectivization using Proposition \ref{prp:to-injective} to obtain C*-correspondences that are still proper and full, and are now also injective by design. Thus, we obtained a strong shift equivalence implemented by \emph{regular} and \emph{full} intermediary C*-correspondences.
\end{proof}

In the work of Kakariadis and Katsoulis \cite{KK-JFA2014} a critical error in the proof of \cite[Theorem 5.9]{KK-JFA2014} (see specifically \cite[Theorem 5.8]{KK-JFA2014}) regarding the preservation of SE in the passage to Pimsner dilations was found (see \cite{CDE} and the corrigendum for \cite{KK-JFA2014}). Unfortunately, the proof of \cite[Theorems 5.9]{KK-JFA2014} contains an additional gap (which appears in the proofs of \cite[Theorem 5.5 \& 7.1]{KK-JFA2014}, as well as in the erratum to \cite{KK-JFA2014}), where it is claimed that if two regular (or both regular and full) C*-correspondences are SSE, then so are their Pimsner dilations. The issue is that the use of \cite[Theorem 5.3]{KK-JFA2014} in the proof of \cite[Theorem 5.5]{KK-JFA2014} only applies when the intermediary C*-correspondences implementing a strong shift equivalence are regular. 

Similar gaps have unfortunately appeared in the proofs of \cite[Theorems~6.3 \& 6.7]{Ery22} and \cite[Theorem~3.5]{BMA} (see the first line of each proof), where it was independently claimed that if two regular C*-correspondences $X$ and $Y$ are strong shift equivalent then the corresponding Cuntz-Pimsner algebras $\mathcal{O}_X$ and $\mathcal{O}_Y$ are equivariantly Morita equivalent in the sense of Combes \cite{Combes94}. Theorem~\ref{thm:reg-reduction} closes these gaps appearing in the proofs of these various aforementioned results.

\begin{rmk} \label{rmk:equiv-stab-iso}
Under the assumption of $\sigma$-unitality of the coefficient C*-algebras, as well as regularity and fullness of the underlying C*-correspondences, by using Theorem~\ref{thm:reg-reduction} and the bicategorical techniques developed in \cite[Sections 4 and 5]{BDR+}, one can prove a result that strengthens both \cite[Theorem 7.1]{KK-JFA2014}, as well as \cite[Theorems~6.3 \& 6.7]{Ery22} and \cite[Theorem~3.5]{BMA}. This also provides the first complete proof which answers a question posed by Muhly, Pask and Tomforde in \cite[Remark 5.5]{MPT08} (see \cite[Conjecture 1]{KK-JFA2014} and \cite[Theorem 3.17]{CDE}).

More precisely, assume that $X$ and $Y$ are regular and full C*-correspondences over $\sigma$-unital coefficient C*-algebras $\A$ and $\B$. We can now show that if $X$ and $Y$ are SSE, then their Cuntz-Pimsner C*-algebras $\mathcal{O}_X$ and $\mathcal{O}_Y$ are stably equivariantly $*$-isomorphic (which by \cite[Theorem 3.17]{CDE} is equivalent to strong Morita equivalence of the Pimsner dilations $X_{\infty}$ and $Y_{\infty}$).

Indeed, by Theorem~\ref{thm:reg-reduction} we may assume that the strong shift equivalence of $X$ and $Y$ is implemented through regular and full intermediary C*-correspondences. Thus, without loss of generality $X$ and $Y$ are elementary strong shift related. It is then straightforward to show that in this case $X$ and $Y$ are aligned shift equivalent in the sense of \cite[Definition 3.1]{BDR+}. Hence, we may apply \cite[Theorem 4.5]{BDR+} and proceed as in the proof of \cite[Theorem 6.3]{BDR+} to conclude.
\end{rmk}

\section{SSE of graph correspondences implies SSE of matrices}\label{sec-modSSE-implies-SSE}

In this section we determine all C*-correspondences over C*-subalgebras of compact operators, and use this to show that SSE of graph C*-correspondences implies SSE of their underlying adjacency matrices. Let $V$ and $W$ be sets, and denote $\A = \oplus_{v \in V} \bbK(\H_v)$ and $\B = \oplus_{w\in W} \bbK(\H_w)$. It is well-known that this is the general form of any C*-subalgebra of compact operators on Hilbert space.

Suppose $\H$ is a Hilbert space. Then $\H$ is naturally a right Hilbert $\bbK(\H)$-module with inner product given by $[ \xi,\eta ] = \theta_{\xi,\eta}$ and right action defined by $\nu \cdot \theta_{\xi , \eta} = \theta_{\eta, \xi} (\nu)$.  This right Hilbert $\bbK(\H)$-module is sometimes called the \emph{standard} $\bbK(\H)$-module.  Since we use the same notation $\H$, to denote the Hilbert space and the right Hilbert $\bbK(\H)$-module, we have two notions of ``compact operators'' on $\H$.  When we write $\bbK(\H)$, we are considering $\H$ as a Hilbert space and $\bbK(\H)$ is the set of compact operators on the Hilbert space $\H$ in the usual sense.  When we write $\K(\H)$, we are considering $\H$ as the standard right $\bbK(\H)$-module and $\K(\H)$ is the set of compact operators on the standard $\bbK(\H)$-module $\H$.

It turns out that $\H$ becomes as imprimitivity $\mathbb{C}-\bbK(\H)$-bimodule since $\K(\H) \cong \mathbb{C}$.

\begin{prp}[Standard $\mathbb{C}-\bbK(\H)$-correspondence] \label{prp-hilbertspace-corr}
Let $\H$ be a Hilbert space with its right $\mathbb{K}(\H)$ module structure. Then $\mathbb{C} \cong \K(\H)$.
\end{prp}

\begin{proof}
Let $\xi_1, \xi_2 \in \H$.  Then $\theta_{\xi_1, \xi_2}(\eta) = \xi_1 \cdot [ \xi_2, \eta ] = \xi_1 \cdot \theta_{\xi_2, \eta} = \theta_{\eta, \xi_2} (\xi_1) = \eta \langle \xi_2, \xi_1 \rangle$.  Therefore, since $\langle \xi_2, \xi_1 \rangle \in \mathbb{C}$ we get that $\K(\H) = \mathbb{C} \cdot \mathrm{Id}$, so that $\mathbb{C} \cong \mathcal{K}(\H)$.
\end{proof}

Let $\H$ and $\H'$ be Hilbert spaces.  Note that $\H'$ is naturally a $\bbK(\H')-\mathbb{C}$-correspondence and by Proposition~\ref{prp-hilbertspace-corr}, $\H$ is a $\mathbb{C}-\bbK(\H)$-correspondence.  Thus we may consider the $\bbK(\H')-\bbK(\H)$-correspondence given by the interior tensor product $\H' \otimes_\mathbb{C} \H$.  Since $\mathbb{C} \cong \K(\H)$, by \cite[Proposition~4.7]{Lan95}, the left $\bbK(\H')$ action on $\H' \otimes_\mathbb{C} \H$ acts faithfully as compact operators on $\H' \otimes_\mathbb{C} \H$, and since 
\begin{align*}
\theta_{\xi_1 \otimes \eta_1, \xi_2 \otimes \eta_2 } ( \xi \otimes \eta) &= ( \xi_1 \otimes \eta_1) \cdot ( \langle \xi_2, \xi \rangle \theta_{\eta_2, \eta}) = \xi_1 \langle \xi_2, \xi \rangle \otimes \theta_{\eta, \eta_2} ( \eta_1) \\
	&= \xi_1 \langle \xi_2, \xi \rangle \otimes \eta \langle \eta_2, \eta_1 \rangle = \xi_1 \langle \xi_2, \xi \rangle \langle \eta_2, \eta_1 \rangle \otimes \eta \\
	&= \theta_{ \xi_1 \langle \eta_2, \eta_1 \rangle, \xi_2 } \cdot ( \xi \otimes \eta),
\end{align*}
the left $\bbK(\H')$ action gives an isomorphism from $\bbK(\H')$ to $\K(\H' \otimes_\mathbb{C} \H)$.  By \cite[Proposition~3.8]{RW-Morita-Eq}, $\H' \otimes_\mathbb{C} \H$ is a $\bbK(\H')-\bbK(\H)$ imprimitivity bimodule with left inner product given by 
\[
[ \xi_1 \otimes \eta_1, \xi_2 \otimes \eta_2 ]  = \theta_{ \xi_1 \langle \eta_2, \eta_1 \rangle, \xi_2 }.
\]
We will call $\H' \otimes_\mathbb{C} \H$ the \emph{standard} $\bbK(\H') - \bbK(\H)$-bimodule.

\begin{prp}\label{pr:imprimitive-bimodule}
Let $\H$ and $\H'$ be Hilbert spaces.  If $X$ is an imprimitivity $\bbK(\H')-\bbK(\H)$-bimodule, then $X \cong \bbK(\H, \H')$.  In particular, $X \cong \H' \otimes_\mathbb{C} \H$.
\end{prp}

\begin{proof}
By \cite[Lemma~3]{Schweizer-PAMS1999}, there exists a Hilbert space $\H''$ such that $\bbK(\H) \cong \bbK(\H'')$ and $X \cong \K( \H'', \H')$.  Since $\bbK(\H) \cong \bbK(\H'')$, by basic representation theory of compact operators, there exists a unitary $U \colon \H \to \H''$ implementing the isomorphism $\bbK(\H) \cong \bbK(\H'')$.  The unitary $U$ induces an unitary isomorphism between $\bbK(\H, \H')$ and $\bbK(\H'', \H')$.  Thus, $X \cong \bbK(\H, \H')$.
\end{proof}

\begin{prp} \label{pr:stnd-cor-decomp}
Let $X$ be a (non-degenerate) $\bbK(\H_v)-\bbK(\H_w)$ correspondence. Then there is some cardinality $\mathsf{C}_{vw}$ such that $X$ and $[\H_v\otimes_\mathbb{C} \H_w]^{ \oplus \mathsf{C}_{vw}}$ are unitarily isomorphic $\bbK(\H_v)-\bbK(\H_w)$ correspondences.
\end{prp}

\begin{proof}
Suppose that $X$ is a $\bbK(\H_v)-\bbK(\H_w)$ correspondence. By \cite[Theorem 1]{Mag97} every $\bbK(\H_v)-\bbK(\H_w)$ submodule of $X$ is complemented. Thus, by generating a sub-bimodule from a non-zero vector, up to unitary isomorphism we may decompose by transfinite induction $X \cong \oplus_{\alpha} X_{\alpha}$ with $0 \leq \alpha < \mathsf{C}_{vw}$ for some cardinality $\mathsf{C}_{vw}$, so that each $X_{\alpha}$ is an irreducible $\bbK(\H_v)-\bbK(\H_w)$ correspondence. We denote by $\phi_{X,\alpha}$ the restriction of the left action $\phi_X$ of $X$ to the $\alpha$-th reducing $\bbK(\H_v)-\bbK(\H_w)$-submodule $X_{\alpha}$ for $0\leq \alpha < \mathsf{C}_{vw}$.

Let $e$ be a minimal projection in $\bbK(\H_w)$.  By \cite[Theorem~5 and 6]{BG02}, the subspace $X_{\alpha,e}:= X_{\alpha}\cdot e$ naturally forms a Hilbert space, and we may identify $\L(X_{\alpha})$ and $\K(X_\alpha)$ up to $*$-isomorphisms with $\bbB(X_{\alpha,e})$ and $\bbK(X_{\alpha, e})$  respectively by the restriction map $T \mapsto T|_{X_{\alpha,e}}$.  Since $X_\alpha$ is an irreducible $\bbK(\H_v)-\bbK(\H_w)$ correspondence, the representation $\Psi$ of $\bbK(\H_v)$ on the Hilbert space $X_{\alpha, e}$ given by $\phi_{X, \alpha} (a) \vert_{X_{\alpha, e} }$ is irreducible.  By basic representation theory of the compact operators, $\Psi(\bbK(\H_v)) = \bbK(X_{\alpha, e} )$ which implies $\phi_{X, \alpha} ( \bbK(\H_v)) = \K(X_{\alpha})$ by the above.  Thus, $\phi_{X, \alpha } \colon \bbK(\H_v) \to \K(X_{\alpha})$ is an isomorphism, and $X_\alpha$ is a $\bbK(\H_v)-\bbK(\H_w)$ imprimitivity bimodule.  By Proposition~\ref{pr:imprimitive-bimodule}, $X_\alpha \cong \H_v \otimes_\mathbb{C} \H_w$.  Consequently, $X$ and $[\H_v\otimes_\mathbb{C} \H_w]^{(\mathsf{C}_{vw})}$ are unitarily isomorphic $\bbK(\H_v)-\bbK(\H_w)$ correspondences.
\end{proof}

Recall that $\A= \oplus_{v\in V}\bbK(\H_v)$ and $\B=\oplus_{w\in W}\bbK(\H_w)$ are arbitrary subalgebras of compact operators. In what follows, we will identify $f \in \oplus_v \bbK(\H_v)$ and $g \in \oplus_w \bbK(\H_w)$ as functions on $V$ and $W$ respectively such that $f(v)\in \bbK(\H_v)$ and $g(w) \in \bbK(\H_w)$ for $v\in V$ and $w\in W$. Let $\mathsf{C}$ be a $V\times W$ matrix so that $\mathsf{C}_{vw}$ is some cardinal entry. We denote as before
$$
E_\mathsf{C}:= \{ \ (v,\alpha,w) \ | \ 0\leq \alpha < \mathsf{C}_{vw}, \ v\in V, \ w\in W \ \}.
$$

With the data above, we may construct an $\A-\B$ correspondence $X(\mathsf{C})$. First we let $X_0(\mathsf{C})$ be all finitely supported functions $\xi$ on $E_\mathsf{C}$ such that $\xi(v,\alpha,w) \in \H_v \otimes_\mathbb{C} \H_w$ is an element of the $\bbK(\H_v)-\bbK(\H_w)$ correspondence $\H_v \otimes_\mathbb{C} \H_w$. By taking the Hausdorff completion of $X_0(\mathsf{C})$ with respect to the norm induced by the inner product
$$
\langle \xi,\eta \rangle(w) :=\sum_{(v,\alpha,w) \in E_\mathsf{C}}\langle \xi(v,\alpha,w),\eta(v,\alpha,w) \rangle_{\bbK(\H_w)},
$$
we obtain a Hilbert C*-module $X(\mathsf{C})$ over $\B$. The left action of $\A$ and the right action of $\B$ are then given by
$$
(f\cdot \xi \cdot g)(v,\alpha,w) := f(v) \cdot \xi(v,\alpha, w) \cdot g(w).
$$

The next result, which is a generalization of \cite[Proposition 3.2]{CDE}, shows that the above construction describes \emph{all} $\A-\B$-correspondences where $\A$ and $\B$ are C*-subalgebras of compact operators.

\begin{prp} \label{pr:corovercmpt}
Let $\A = \oplus_v \bbK(\H_v)$ and $\B = \oplus_w \bbK(\H_w)$ for two sets $V$ and $W$ and Hilbert spaces $\{\H_v\}_{v\in V}$ and $\{\H_w\}_{w\in W}$. Suppose that $X$ is an $\A-\B$ correspondence. Then there exists a \emph{unique} $V\times W$ matrix $\mathsf{C}$ with cardinal entries such that $X$ is unitarily equivalent to $X(\mathsf{C})$ as an $\A-\B$ correspondence.
\end{prp}

\begin{proof}
Let $v\in V$ and $w\in W$, and let $P_v \in \M(\oplus_{v\in V} \bbK(\H_v)) \cong \prod_{v\in V} \bbB(\H_v)$ and $P_w \in \M(\oplus_{w\in W} \bbK(\H_w)) \cong \prod_{w\in W} \bbB(\H_w)$ be the projection onto the $v$-th and $w$-th factors respectively. It follows that $P_vX P_w$ is a $\bbK(\H_v)-\bbK(\H_w)$ correspondence. By Proposition \ref{pr:stnd-cor-decomp} there exists a cardinal $\mathsf{C}_{vw}$ such that $[\H_v \otimes_\mathbb{C} \H_w]^{(\mathsf{C}_{vw})} \cong P_vX P_w$ up to a $\bbK(\H_v)-\bbK(\H_w)$ correspondence unitary $U_{vw}$, which restricts to an isometry $U_{v,\alpha, w}$ on the $\alpha$-th copy of $\H_v \otimes_\mathbb{C} \H_w$. Thus, we may define a map $U : X_c(\mathsf{C}) \rightarrow X$ by setting $U(h \delta_{(v,\alpha, w)}) := U_{v,\alpha, w}(h)$, where $h \delta_{(v,\alpha, w)}$ is the function on $E_\mathsf{C}$ with all zeros, except that its value at the point $(v,\alpha, w) \in E_\mathsf{C}$ is the element $h \in \H_v \otimes_\mathbb{C} \H_w$, and then extend $U$ to a $\A-\B$ bimodule map on $X_c(\mathsf{C})$. But now, if $\xi \in X_c(\mathsf{C})$ is a finitely supported function such that $\xi(v,\alpha, w) \in \H_v\otimes_\mathbb{C} \H_w$, we get that
$$
|U(\xi)|^2(w)= \sum_{(v,\alpha, w) \in E_\mathsf{C}} |U_{v,\alpha, w}(\xi(v,\alpha, w))|^2 = \sum_{(v,\alpha, w) \in E_\mathsf{C}} |\xi(v,\alpha, w)|^2 = |\xi|^2(w).
$$
Hence, $U$ extends to an isometric bimodule map on $X(\mathsf{C})$. Since the linear span of elements $U_{vw}(h \delta_{(v,\alpha, w)})$, varying over $h\in \H_v \otimes_\mathbb{C} \H_w$, $v\in V, w\in W$ and $0 \leq \alpha < \mathsf{C}_{vw}$, is dense in $X$, we get that $U$ is surjective. Therefore, $U$ is a unitary isomorphism of $\A-\B$ correspondences.

As for uniqueness, suppose that $Y$ is another $\A-\B$ correspondence which is unitarily isomorphic to $X$ via an $\A-\B$-module unitary $U$. Hence, if $\mathsf{C}^X$ and $\mathsf{C}^Y$ are the matrices associated with $X$ and $Y$ respectively, by the first part of the proof we must have that $\mathsf{C}^X_{vw} = \mathsf{C}^Y_{vw}$ for every $v\in V$ and $w\in W$.
\end{proof}

\begin{prp} \label{pr:mult-to-tensor}
let $U,V,W$ be sets, and let $\A = \oplus_{u\in U} \bbK(\H_u)$, $\B = \oplus_{v\in V} \bbK(\H_v)$ and $\C= \oplus_{w\in W} \bbK(\H_w)$ be C*-subalgebras of compact operators. Suppose $\mathsf{C}$ and $\mathsf{D}$ are $U\times V$ and $V\times W$ matrices (respectively) with cardinal entries. Let $X(\mathsf{C})$, $X(\mathsf{D})$ and $X(\mathsf{C}\mathsf{D})$ be the $\A-\B$, $\B-\C$ and $\A-\C$ correspondences (respectively) as constructed above. Then, there is a unitary isomorphism $\tau: X(\mathsf{C})\otimes_{\B} X(\mathsf{D}) \rightarrow X(\mathsf{C}\mathsf{D})$.
\end{prp}

\begin{proof}
It is clear that both $X(\mathsf{C}) \otimes_\B X(\mathsf{D})$ and $X(\mathsf{C}\mathsf{D})$ are $\A-\C$ correspondences, so by uniqueness in Proposition \ref{pr:corovercmpt} it suffices to show that for every $u\in U$ and $w\in W$, the multiplicity of $\H_u\otimes_\mathbb{C} \H_w$ in $P_u X(\mathsf{C}) \otimes_\B X(\mathsf{D}) P_w$ is exactly equal to $(\mathsf{C}\mathsf{D})_{uw}$.

So fix $u\in U$ and $w\in W$. We write for every $v\in V$ the $\bbK(\H_u)-\bbK(\H_v)$ correspondence $P_u X(\mathsf{C}) P_v$, and the $\bbK(\H_v) - \bbK(\H_w)$ correspondence $P_v X(\mathsf{D}) P_w$. It is easy to verify that $P_u X(\mathsf{C}) \otimes_{\B} X(\mathsf{D}) P_w \cong \oplus_{v\in V} P_u X(\mathsf{C}) P_v \otimes_{\bbK(\H_v)} P_v X(\mathsf{D}) P_w$, and by assumption we have the unitary isomorphisms $P_u X(\mathsf{C}) P_v \cong [\H_u \otimes_\mathbb{C} \H_v]^{(\mathsf{C}_{uv})}$ and $P_v X(\mathsf{D}) P_w \cong [\H_v\otimes_\mathbb{C} \H_w]^{(\mathsf{D}_{vw})}$. A straightforward verification then shows that
$$
P_u X(\mathsf{C}) P_v \otimes_{\bbK(\H_v)} P_v X(\mathsf{D}) P_w \cong [\H_u \otimes_\mathbb{C} \H_v]^{(\mathsf{C}_{uv})} \otimes_{\bbK(\H_v)} [\H_v\otimes \H_w]^{(\mathsf{D}_{vw})} \cong [\H_u \otimes_\mathbb{C} \H_w]^{(\mathsf{C}_{uv}D_{vw})}.
$$
Thus, the multiplicity of $\H_u \otimes_\mathbb{C} \H_w$ in $P_u X(\mathsf{C}) \otimes_{\B} X(\mathsf{D}) P_w$ is exactly equal to the cardinality of a disjoint union over $v\in V$ of sets, each of cardinality $\mathsf{C}_{uv}\mathsf{D}_{vw}$. This latter cardinality is equal by definition to $(\mathsf{C}\mathsf{D})_{uw}$.
\end{proof}

The following is the main result of this section, leading to a resolution of an open conjecture from \cite[Section 6]{KK-JFA2014} (see the paragraph right before \cite[Theorem 6.1]{KK-JFA2014}).

\begin{thm}\label{thm-SSEmodule-implies-SSE}
Let $\mathsf{A}$ and $\mathsf{B}$ be finite essential matrices with entries in $\mathbb{N}$ over $V$ and $W$ respectively. Denote $\A = c(V)$ and $\B = c(W)$. If $X(\mathsf{A})$ and $X(\mathsf{B})$ are strong shift equivalent with lag $m$, then $\mathsf{A}$ and $\mathsf{B}$ are strong shift equivalent with lag $m$.
\end{thm}

\begin{proof}
Since $X(\mathsf{A})$ and $X(\mathsf{B})$ are strong shift equivalent, there exist C*-correspondences $R_1,...,R_m$ and $S_1,...,S_m$ such that
\begin{equation} \label{eq:modsse}
X(\mathsf{A}) \cong R_1 \otimes S_1, S_1 \otimes R_1 \cong R_2 \otimes S_2, \ ... \ , S_{m-1}\otimes R_{m-1} \cong R_m \otimes S_m, S_m \otimes R_m \cong X(\mathsf{B}).
\end{equation}
We denote by $\C_n$ the coefficient C*-algebra of $S_n\otimes R_n \cong R_{n+1} \otimes S_{n+1}$ for $n=1,...,m-1$. Since $\mathsf{A}$ and $\mathsf{B}$ are essential, it follows that $X(\mathsf{A})$ and $X(\mathsf{B})$ are regular, so that by Theorem \ref{thm:reg-reduction} we may assume without loss of generality that $R_1,...,R_m$ and $S_1,...,S_m$ are regular.

Now, since $c_0(W)$ is a C*-subalgebra of compact operators, by \cite[Theorem~9(i)]{BG02}, $\K(R_m)$ is also a C*-subalgebra of compact operators, and since $R_m$ is regular, it follows that $\C_{m-1}$ is a C*-subalgebra of compact operators. We continue this by induction, so that if $\C_n$ is a C*-subalgebra of compact operators, then  $\K(R_n)$ is as well (by \cite[Theorem~9(i)]{BG02}), and since $R_n$ is regular it follows that $\C_{n-1}$ is also a C*-subalgebra of compact operators. Thus, we have shown that $\C_n$ are all C*-subalgebras of compact operators for $n=1,...,m-1$.

Thus, by Proposition \ref{pr:corovercmpt} there exist sets $V=U_0, U_1,...,U_{m-1},U_m = W$, as well as $U_i \times U_{i+1}$ matrices $\mathsf{C}_i$ and $U_{i+1} \times U_i$ matrices $\mathsf{D}_i$ such that $R_i \cong X(\mathsf{C}_i)$ and $S_i \cong X(\mathsf{D}_i)$ for $i=1,...,m$. Now, by equation \eqref{eq:modsse}, Proposition \ref{pr:mult-to-tensor} and uniqueness in Proposition \ref{pr:corovercmpt} we get that
$$
\mathsf{A} = \mathsf{C}_1\mathsf{D}_1, \ \mathsf{D}_1\mathsf{C}_1 = \mathsf{C}_2\mathsf{D}_2, \ ... \ , \mathsf{D}_{m-1}\mathsf{C}_{m-1} = \mathsf{C}_m\mathsf{D}_m, \ \mathsf{D}_m\mathsf{C}_m = \mathsf{B}.
$$
Hence, to show that $\mathsf{A}$ and $\mathsf{B}$ are strong shift equivalent, we need only show that $\mathsf{C}_i$ and $\mathsf{D}_i$ can be chosen to be finite and with entries in $\mathbb{N}$ for $i=1,...,m$. We do this by induction. Since $\mathsf{A}$ is finite with finite entries, we only need finitely many columns of $\mathsf{C}_1$ and finitely many rows of $\mathsf{D}_1$ to get $\mathsf{A}=\mathsf{C}_1\mathsf{D}_1$ and each non-finite entry can be made into a $0$. Hence, if we remove the unnecessary columns / rows of $\mathsf{C}_1$ and $\mathsf{D}_1$ we get (respectively) finite matrices $\mathsf{C}_1'$, which is $V\times U_1'$, and $\mathsf{D}_1'$, which is $U_1'\times V$, such that we still have $\mathsf{A}=\mathsf{C}_1'\mathsf{D}_1'$. Then, the matrix $\mathsf{D}_1'\mathsf{C}_1'$ coincides with the $U_1'\times U_1'$ corner of $\mathsf{C}_2\mathsf{D}_2$, so we may compress $\mathsf{C}_2$ to a $U_1' \times U_2$ matrix, and $\mathsf{D}_2$ to a $U_2 \times U_1'$ matrix. Thus, we may assume that $\mathsf{C}_2$ and $\mathsf{D}_2$ are $U_1'\times U_2$ and $U_2\times U_1'$ respectively, where $U_1'$ is finite. We may then proceed by induction to obtain \emph{finite} sets $V = U_0',U_1',...,U_{m-1}',U_m'$ together with $U_i' \times U_{i+1}'$ matrices $\mathsf{C}_i'$ and $U_{i+1}' \times U_i'$ matrices $\mathsf{D}_i'$ with entries in $\mathbb{N}$ such that
$$
\mathsf{A} = \mathsf{C}_1'\mathsf{D}_1', \ \mathsf{D}_1'\mathsf{C}_1' = \mathsf{C}_2'\mathsf{D}_2', \ ... , \ \mathsf{D}_{m-1}'\mathsf{C}_{m-1}' = \mathsf{C}_m'\mathsf{D}_m'.
$$
However, now the issue is that $U_m'$, despite being finite, might not coincide with $W$. So we repeat the procedure above from the end to the beginning to produce finite sets $W=U_m'',U_{m-1}'',...,U_1'',U_0$, together with matrices $U_{i+1}'' \times U_i''$ matrices $\mathsf{D}_i''$ and $U_i'' \times U_{i+1}''$ matrices $\mathsf{C}_i''$ such that
$$
\mathsf{B} = \mathsf{D}_m''\mathsf{C}_m'', \ \mathsf{C}_m''\mathsf{D}_m'' = \mathsf{D}_{m-1}''\mathsf{C}_{m-1}'', \ ... , \ \mathsf{C}_2''\mathsf{D}_2'' = \mathsf{D}_1''\mathsf{C}_1''.
$$
If we take the finite sets $\widehat{U}_i:= U_i' \cup U_i''$, and take the $\widehat{U}_i \times \widehat{U}_{i+1}$ compression $\widehat{\mathsf{C}}_i$ of $\mathsf{C}_i$ for $i=1,...,m$, as well as the  $\widehat{U}_{i+1} \times \widehat{U}_i$ compression $\widehat{\mathsf{D}}_i$ of $\mathsf{D}_i$ for $i=1,...,m$, we would get that $\widehat{U}_0 = V$ and $\widehat{U}_m=W$, as well as
$$
\mathsf{A} = \widehat{\mathsf{C}}_1\widehat{\mathsf{D}}_1, \ \widehat{\mathsf{D}}_1\widehat{\mathsf{C}}_1 = \widehat{\mathsf{C}}_2\widehat{\mathsf{D}}_2, \ ... \ , \widehat{\mathsf{D}}_{m-1}\widehat{\mathsf{C}}_{m-1} = \widehat{\mathsf{C}}_m\widehat{\mathsf{D}}_m, \ \widehat{\mathsf{D}}_m\widehat{\mathsf{C}}_m = \mathsf{B},
$$
where now all matrices $\widehat{\mathsf{C}}_i$ and $\widehat{\mathsf{D}}_i$ are finite.
\end{proof}

Since the converse of the above theorem is also true by \cite[Proposition 3.3]{CDE}, this shows that the shift equivalence problem for C*-correspondences from \cite[Section 6]{KK-JFA2014} has a negative solution, again due to the $7 \times 7$ irreducible counterexample of Kim and Roush \cite{KR99}.

\begin{cor} \label{cor:sep-neg}
There exist $7\times 7$ irreducible matrices $\mathsf{A}$ and $\mathsf{B}$ with entries in $\mathbb{N}$ such that $X(A)$ and $X(B)$ are shift equivalent but not strong shift equivalent.
\end{cor}

We now use our theorem to show that elementary strong shift relation of C*-correspondences is not a transitive relation. This answers another question in \cite{KK-JFA2014} (see the paragraph right after \cite[Definition 4.3]{KK-JFA2014}).  

\begin{exm}
Let $\mathsf{A} = \begin{bmatrix} 1 & 1 & 0 \\ 0 & 0 & 1 \\ 1 & 1 & 1 \end{bmatrix}$, $\mathsf{B} = \begin{bmatrix} 1& 1 \\ 1 & 1 \end{bmatrix}$, and $\mathsf{C} = \begin{bmatrix} 2 \end{bmatrix}$. By \cite[Example 7.2.2]{LM95} we know that $\mathsf{A}$ is elementary strong shift related to $\mathsf{B}$ and $\mathsf{B}$ is elementary strong shift related to $\mathsf{C}$ but $\mathsf{A}$ is not elementary strong shift related to $\mathsf{C}$. By Theorem~\ref{thm-SSEmodule-implies-SSE}, $X(\mathsf{A})$ is elementary strong shift related to $X(\mathsf{B})$ and $X(\mathsf{B})$ is elementary strong shift related to $X(\mathsf{C})$ but $X(\mathsf{A})$ is not elementary strong shift related to $X(\mathsf{C})$.  Therefore, elementary strong shift relation for C*-correspondence is not a transitive relation.  
\end{exm}


\begin{thebibliography}{19}


\bibitem{BG02} Damir Baki\'c and Boris Gulja\v s, \emph{Hilbert C*-modules over C*-algebras of compact operators}, Acta Sci. Math (Szeged), \textbf{68} (2002), 249-269.

\bibitem{BDR+}
Boris Bilich, Adam Dor-On and Efren Ruiz, \emph{Equivariant homotopy classification of graph C*-algebras}, arXiv preprint: 2408.09740

\bibitem{RB75} Rufus Bowen, \emph{Equilibrium states and the ergodic theory of Anosov diffeomorphisms}, Springer Lecture Notes in Math., \textbf{470} (1975), Springer-Verlag.


\bibitem{BMA} Kevin Brix,  Alexander Mundey, and Adam Rennie, \emph{Splittings for C*-correspondences and strong shift equivalence}, Math. Scand. \textbf{130} (2024), no. 1, 101--148.


\bibitem{CDE} Toke Meier Carlsen, Adam Dor-On, and S{\o}ren Eilers, \emph{Shift equivalence through the lens of Cuntz-Krieger algebras}, Anal. PDE \textbf{17} (2024), no. 1, 345--377.

\bibitem{CR} Toke Meier Carlsen and James Rout, \emph{Diagonal-preserving gauge-invariant isomorphisms of graph C*-algebras}, J. Funct. Anal., \textbf{273} (2017), 2981--2993.

\bibitem{Combes94} F. Combes, \emph{Crossed products and Morita equivalence}, Proc. London Math. Soc. (3) \textbf{49} (1984), no. 2, 289--306.


\bibitem{Cu81}
Joachim Cuntz, \emph{A class of {$C^{\ast} $}-algebras and topological {M}arkov chains II. Reducible chains and the {E}xt-functor for {$C^{\ast} $}-algebras}, Invent. Math. \textbf{63} (1981), 25--40.

\bibitem{CK80}
Joachim Cuntz and Wolfgang Krieger, \emph{A class of {$C^{\ast} $}-algebras and topological {M}arkov chains},
Invent. Math. \textbf{56} (1980), 251--268.

\bibitem{DGS76} Manfred Denker, Christian Grillenberger, and Karl Sigmund, \emph{Ergodic theory on compact spaces}, Springer Lecture Notes in Math.\textbf{527} (1976).


\bibitem{Ery22}  Menev\c{s}e Ery\"{u}zl\"{u}, \emph{Passing {C*}-correspondence Relations to the Cuntz-Pimsner algebras}, M\"{u}nster J. Math. \textbf{15} (2022), no. 2, 441--471.

\bibitem{FMR-IUMJ2003} Neal J. Fowler, Paul S. Muhly, and Iain Raeburn, \emph{Representations of Cuntz-Pimsner Algebras}, Indiana Univ. Math. J. No. 3, \textbf{52} (2003), 569--605.

\bibitem{KK-JFA2014} Evgenios T.A. Kakariadis and Elias G. Katsoulis, \emph{C*-algebras and equivalences for C*-correspondences}, Journal of Funct. Anal. \textbf{266} (2014), 956--988.


\bibitem{KatII} Takeshi Katsura, \emph{Ideal structure of {C*}-algebras associated with {C*}-correspondences}, Pacific J. Math. \textbf{230} (2007), 107--145.


\bibitem{KRdeci} Ki Hang Kim and Fred W. Roush. \emph{Some results on decidability of shift equivalence}, J. Combin. Inform. System Sci., 4(2):123-146, 1979.

\bibitem{KRdecii} Ki Hang Kim and Fred W. Roush. Decidability of shift equivalence. In Dynamical systems (College Park, MD, 1986-87), volume 1342 of Lecture Notes in Math., pages 374-424. Springer, Berlin, 1988


\bibitem{KR99}
		Ki Hang Kim and Fred W. Roush, \emph{The {W}illiams conjecture is false for irreducible subshifts}, Ann. of Math. (2) \textbf{149} (1999), 545--558.


\bibitem{Lan95} E. Christopher Lance,
	\emph{Hilbert C*-modules. A toolkit for operator algebraists}, London Mathematical Society Lecture Note Series, \textbf{210}. Cambridge University Press, Cambridge, 1995.
	
\bibitem{LM95} Douglas Lind and Brian Marcus,
\textit{An introduction to symbolic dynamics and coding},
Cambridge University Press, Cambridge, 1995.
	
\bibitem{Mag97}
Bojan Magajna, \emph{Hilbert C*-modules in which all closed submodules are complemented}, Proc. Amer. Math. Soc. \textbf{125} (1997),  849-852.

\bibitem{MPT08} Paul S.~Muhly, David Pask, and Mark Tomforde \emph{Strong shift equivalence of  C*-correspondences}, Israel J. Math., \textbf{167} (2008), 315--346.

\bibitem{Raeburn-graph}
Iain Raeburn, \emph{Graph algebras}, CBMS Regional Conference Series in Mathematics \textbf{103} (2005), vi+113.

 \bibitem{RW-Morita-Eq} Iain Raeburn and Dana P.~Williams, \emph{Morita equivalence and Continuous-Trace {C*}-algebras}, Series: Mathematical Surveys and Monograph, AMS, no. 60, 1998.

\bibitem{Schweizer-PAMS1999} J\"{u}rgen Schweizer, \emph{A description of Hilbert {C*}-modules in which all closed submodules are orthogonally closed}, Proc. Amer. Math. Soc. \textbf{127} (1999),  2123-2125.

\bibitem{Wil73}
Robert F. Williams, \emph{Classification of subshifts of finite type}, Ann. of Math. (2) \textbf{98} (1973), 120--153; errata, ibid. (2) 99 (1974), 380--381.

\end{thebibliography}
\end{document}